\documentclass{article}
\usepackage[colorlinks,citecolor=blue,urlcolor=blue,filecolor=blue,backref=page]{hyperref}
\usepackage[margin=3.5cm]{geometry}

\usepackage{amsthm}
\usepackage{amsmath}
\usepackage[round]{natbib}

\usepackage{graphicx}
\usepackage{enumerate}
\DeclareRobustCommand{\VANDER}[3]{#2}
\usepackage[T1]{fontenc}
\usepackage[utf8]{inputenc}
\usepackage{amsmath,amsfonts}
\usepackage{amssymb}
\usepackage{bbm}
\usepackage{bm}
\usepackage{mathtools}
\newcommand{\dif}{\,\textnormal{d}}

\newtheorem{theorem}{Theorem}
\newtheorem{corollary}[theorem]{Corollary}
\newtheorem{lemma}[theorem]{Lemma}
\newtheorem{proposition}[theorem]{Proposition}

\newtheorem{definition}{Definition}
\newtheorem{example}{Example}

\renewcommand{\H}{\mathcal H}

\DeclareMathOperator{\F}{\mathcal F}
\DeclareMathOperator{\B}{\mathcal B}

\DeclareMathOperator{\G}{\mathcal G}

\DeclareMathOperator{\C}{\mathcal C}

\newcommand{\PM}{\ensuremath{\bar{P}_0}}
\newcommand{\QM}{\ensuremath{\bar{P}_1}}
\newcommand{\KM}{\ensuremath{\bar{P}_k}}

\newcommand{\initialspace}{\ensuremath{\langle {\cal X}^m \rangle}}
\newcommand{\samplespacen}{\ensuremath{\langle {\cal X}^n \rangle}}
\newcommand{\initialX}{\ensuremath{\langle X^{(m)} \rangle}}
\newcommand{\initialsigma}{\ensuremath{\langle \Sigma^{(m)} \rangle}}

\newcommand{\Pn}[1][n]{\bar{P}_{0}^{(#1)}}
\newcommand{\Qn}[1][n]{\bar{P}_{1}^{(#1)}}
\newcommand{\genericP}[1]{\ensuremath{P_{#1,e}}}
\DeclareMathOperator{\A}{\mathcal{A}}
\DeclareMathOperator{\X}{\mathcal{X}}

\newcommand{\R}{\ensuremath{\mathbb{R}}}
\newcommand{\rnd}[2]{\frac{\dif {#1}}{\dif {#2}}}
\newcommand{\ind}[1][\cdot]{\ensuremath{\mathbbm{1}_{\{#1\}}}}

\newcommand{\mumarginal}[3]{\ensuremath{{\bar{P}_{#1}^{[{#2}_{#3}]}}}}
\newcommand{\mumargcond}[4]{\ensuremath{{\bar{P}_{#1}^{[{#2}_{#3}]}
			(\cdot \mid #4)}}}
\newcommand{\mugeneral}[3]{\ensuremath{{#1}^{[{#2}_{#3}]}}}
\newcommand{\mug}[3]{\ensuremath{{P}_{#1}^{[{#2}_{#3}]}}}
\newcommand{\mugcond}[4]{\ensuremath{{P}_{#1}^{[{#2}_{#3}]}(\cdot \mid #4 )}}
\newcommand{\mugencond}[4]{\ensuremath{{#1}^{[{#2}_{#3}]}(\cdot \mid #4 )}}

\newcommand{\cX}{\ensuremath{\mathcal X}}
\DeclareBoldMathCommand{\AB}{A} \DeclareBoldMathCommand{\vmu}{\mu}
\DeclareBoldMathCommand{\zeros}{0}
\DeclareBoldMathCommand{\SSigma}{\Sigma}

\author{Allard Hendriksen\footnote{CWI, Amsterdam, allard.hendriksen@cwi.nl}, Rianne de Heide\footnote{CWI, Amsterdam and Leiden University, The Netherlands, r.de.heide@cwi.nl pdg@cwi.nl}, Peter Gr\"unwald$^\dag$}
\date{\today}
\title{Optional Stopping with Bayes Factors: a categorization and extension of folklore results,
	with an application to invariant situations}

\begin{document}

\maketitle

\begin{abstract}
  It is often claimed that Bayesian methods, in particular Bayes
factor methods for hypothesis testing, can deal with optional
stopping.  We first give an overview, using elementary probability
theory, of three different mathematical meanings that various
authors give to this claim: (1) stopping rule {\em independence},
(2) posterior {\em calibration\/} and (3) (semi-) {\em frequentist
	robustness to optional stopping}. We then prove theorems to the
effect that these claims do indeed hold in a general
measure-theoretic setting. For claims of type (2) and (3), such
results are new. By allowing for non-integrable measures based on
improper priors, we obtain particularly strong results for the
practically important case of models with nuisance parameters
satisfying a group invariance (such as location or scale). We also
discuss the practical relevance of (1)--(3), and conclude that
whether Bayes factor methods actually perform well under
optional stopping crucially depends on details of models, priors and
the goal of the analysis.
\end{abstract}

\section{Introduction}
In recent years, a surprising number of scientific results have failed
to hold up to continued scrutiny. Part of this `replicability crisis'
may be caused by practices that ignore the assumptions of traditional
(frequentist) statistical methods \citep{john-2012-measur-preval}. One
of these assumptions is that the experimental protocol should be
completely determined upfront.  In practice, researchers often adjust
the protocol due to unforeseen circumstances or collect data until a
point has been proven.  This practice, which is referred to as
\emph{optional stopping}, can cause true hypotheses to be wrongly
rejected much more often than these statistical methods promise.

Bayes factor hypothesis testing has long been advocated as an
alternative to traditional testing that can resolve several of its
problems; in particular, it was claimed early on that Bayesian methods
continue to be valid under optional stopping
\citep{lindley-1957-statis-parad,RaiffaS61,edwards-1963-bayes-statis}. In
particular, the latter paper claims that (with Bayesian methods) ``it
is entirely appropriate to collect data until a point has been proven
or disproven, or until the data collector runs out of time, money, or
patience.'' In light of the replicability crisis, such claims have
received much renewed interest
\citep{wagenmakers-2007-pract-solut,rouder-2014-option,schonbrodt-2017-sequen-hypot,yu-2013-when-decis,sanborn-2013-frequen-implic}. But
what do they mean mathematically?  It turns out that different authors
mean quite different things by `Bayesian methods handle optional
stopping'; moreover, such claims are often shown to hold only in an
informal sense, or in restricted contexts.  Thus, the first goal of
the present paper is to give a systematic overview and formalization
of such claims in a simple, expository setting and, still in this
simple setting, explain their relevance for practice: can we
effectively rely on Bayes factor testing to do a good job under
optional stopping or not? As we shall see, the answer is subtle. The
second goal is to extend the reach of such claims to more general
settings, for which they have never been formally verified and for
which verification is not always trivial.

\paragraph{Overview}
In {\em Section~\ref{sec:simple}}, we give a systematic overview of
what we identified to be the three main mathematical senses in which
Bayes factor methods can handle optional stopping, which we call {\em
	$\tau$-independence, calibration, and (semi-)frequentist}.  We first
do this in a setting chosen to be as simple as possible --- finite
sample spaces and strictly positive probabilities --- allowing for
straightforward statements and proofs of results. In {\em
	Section~\ref{sec:discussion}}, we explain the practical relevance of
these three notions. It turns out that whether or not we can say that
`the Bayes factor method can handle optional stopping' in practice is
a subtle matter, depending on the specifics of the given situation:
what models are used, what priors, and what is the goal of the
analysis.  We can thus explain the paradox that there have also been
claims in the literature that Bayesian methods {\em cannot\/} handle
optional stopping in certain cases; such claims were made, for example
by \cite{yu-2013-when-decis,sanborn-2013-frequen-implic}, and also by
ourselves \citep{heide-2017-why-option}. We also briefly discuss {\em
	safe tests\/} \citep{safetesting} which can be interpreted as a
novel method for determining priors that behave better under frequentist
optional stopping.  The paper has been organized in such a way that
these first two sections can be read with only basic knowledge of
probability theory and Bayesian statistics. For convenience, we
illustrate Section~\ref{sec:discussion} with an informally stated
example involving group invariances, so that the reader gets a
complete overview of what the later, more mathematical sections are
about.

{\em Section~\ref{sec:general}\/} extends the statements and
results to a much more general setting allowing for a wide range of
sample spaces and measures, including measures based on {\em improper
	priors}. These are priors that are not integrable, thus not defining
standard probability distributions over parameters, and as such they
cause technical complications. Such priors are indispensable within
the recently popularized \emph{default Bayes factors} for common
hypothesis tests
\citep{rouder-2009-bayes,rouder-2012-default-bayes,jamil-2016-default-gunel}.

In {\em Section~\ref{sec:group}}, we provide stronger results for the
case in which both models satisfy the same group invariance. Several
(not all) default Bayes factor settings concern such situations;
prominent examples are Jeffreys' (\citeyear{jeffreys-1961-theory-of})
Bayesian one- and two-sample $t$-tests, in which the models are
location and location-scale families, respectively. Many more examples
are given by Berger and various collaborators
\citep{berger1998bayes,dass-2003-unified-condit,bayarri2012criteria,bayarri2016rejection}. These
papers provide compelling arguments for using the (typically improper)
{\em right Haar prior\/} on the nuisance parameters in such
situations; for example, in Jeffreys' one-sample $t$-test, one puts a
right Haar prior on the variance. In particular, in our restricted
context of Bayes factor hypothesis testing, the right Haar prior does
not suffer from the {\em marginalization paradox\/}
\citep{dawid1973marginalization} that often plagues Bayesian inference
based on improper priors (we briefly return to this point in the conclusion).

Haar priors and group invariant
models were studied extensively by
\cite{eaton-1989-group-invar,andersson-1982-distr-maxim,wijsman-1990-invar},
whose results this paper depends on considerably.  When nuisance
parameters (shared by both $H_0$ and $H_1$) are of suitable form and
the right Haar prior is used, we can strengthen the results of
Section~\ref{sec:general}: they now hold uniformly for all possible
values of the nuisance parameters, rather than in the marginal, `on
average' sense we consider in Section~\ref{sec:general}. However ---
and this is an important insight --- we {\em cannot take arbitrary
	stopping rules\/} if we want to handle optional stopping in this
strong sense: our theorems only hold if the stopping rules satisfy a
certain intuitive condition, which will hold in many but not all
practical cases: the stopping rule must be ``invariant'' under some
group action. For instance, a rule such as `stop as soon as the Bayes
factor is $\geq 20$' is allowed, but a rule (in the Jeffreys'
one-sample $t$-test) such as `stop as soon as $\sum x_i^2 \geq 20$' is
not.

The paper ends with supplementary material, comprising Section~\ref{app:groupprel} containing basic background material about groups, and Section~\ref{app:proofs} containing all longer mathematical
proofs.

\paragraph{Scope and Novelty} Our analysis is restricted to Bayesian
testing and model selection using the Bayes factor method; we do not
make any claims about other types of Bayesian inference. Some of the
results we present were already known, at least in simple settings; we
refer in each case to the first appearance in the literature that we
are aware of. In particular, our results in
Section~\ref{sec:generalindependence} are implied by earlier results
in the seminal work by \cite{BergerW88} on the likelihood principle;
we include them any way since they are a necessary building block for
what follows.  The real mathematical novelties in the paper are the
results on calibration and (semi-) frequentist optional stopping with
general sample spaces and improper priors and the results on the group
invariance case
(Section~\ref{sec:generalcalibration}--\ref{sec:group}). These results
are truly novel, and --- although perhaps not very surprising --- they
do require substantial additional work not covered by
\cite{BergerW88}, who are only concerned with $\tau$-independence. In
particular, the calibration results require the notion of the
`posterior odds of some particular posterior odds', which need to be
defined under arbitrary stopping times. The difficulty here is that,
in contrast to the fixed sample sizes where even with
continuous-valued data, the Bayes factor and the posterior odds
usually have a distribution with full support, with variable stopping
times, the support may have `gaps' at which its density is zero or
very near zero. An additional difficulty encountered in the group
invariance case is that one has to define filtrations based on maximal
invariants, which requires excluding certain measure-zero points from the
sample space.
\section{The Simple Case}
\label{sec:simple}
Consider a finite set ${\cal X}$ and a sample space $\Omega \coloneqq {\cal
	X}^T$ where $T$ is some very large (but in this section, still
finite) integer. One observes a {\em sample} $x^{\tau} \equiv x_1, \ldots,
x_{\tau}$, which is an initial segment of $x_1, \ldots, x_T \in \X^T$. In the simplest case, $\tau = n$ is a sample size that is
fixed in advance; but, more generally $\tau$ is a {\em stopping
	time\/} defined by some stopping rule (which may or may not be known
to the data analyst), defined formally below.

We consider a hypothesis testing scenario where we wish to distinguish
between a null hypothesis \(H_0\) and an alternative hypothesis
\(H_1\). Both $H_0$ and $H_1$ are sets of distributions on $\Omega$, and they are each represented by unique probability
distributions \(\PM\) and \(\QM\) respectively.  Usually, these
are taken to be Bayesian marginal distributions,
defined as follows. First one writes, for both $k \in \{0,1\}$, $H_k =
\{P_{\theta \mid k} \mid \theta \in \Theta_k\}$ with `parameter spaces'
$\Theta_k$; one then defines or assumes some prior probability
distributions $\pi_0$ and $\pi_1 $ on $\Theta_0$ and $\Theta_1$,
respectively. The Bayesian marginal probability distributions are then the
corresponding marginal distributions, i.e.\ for any set $A \subset
\Omega$ they satisfy:
\begin{align}\label{eq:bayesfactor}
\PM(A) = \int_{\Theta_0} P_{\theta|0}(A) \dif \pi_0(\theta) \ \ ;\ \
\QM(A) = \int_{\Theta_1} P_{\theta|1}(A) \dif \pi_1(\theta).
\end{align}
For now we also further assume that for every $n\leq T$, every $x^n
\in {\cal X}^n$, $\PM(X^n = x^n) > 0$ and $\QM(X^n = x^n) > 0$ (full
support), where here, as below, we use random variable notation, $X^n =
x^n$ denoting the event $\{x^n \} \subset \Omega$.  We note that there
exist approaches to testing and model choice such as testing by
nonnegative martingales \citep{shafer2011test,PasG18} and minimum
description length \citep{BarronRY98,Grunwald07} in which the $\PM$ and $\QM$ may
be defined in different (yet related) ways. Several of the  results below
extend to general $\PM$ and $\QM$; we return to this point at the end of the paper, in Section~\ref{sec:concluding remarks}. In all
cases, we further assume that we have determined an additional
probability mass function $\pi$ on $\{H_0,H_1\}$, indicating the prior
probabilities of the hypotheses.  The evidence in favor of $H_1$
relative to $H_0$ given data $x^{\tau}$ is now measured either by the
{\em Bayes factor\/} or the {\em posterior odds}.  We now give the standard definition of these quantities
for the case that $\tau = n$, i.e., that the sample size is fixed in
advance. First, noting that all conditioning below is on events of
strictly positive probability, by Bayes' theorem, we can write for any
$A \subset \Omega$,
\begin{align}\label{eq:simple}
\frac{\pi(H_1 \mid A)}{\pi(H_0 \mid A)}
&=   \frac{P(A \mid H_1)}{P(A  \mid H_0)}
\cdot     \frac{\pi(H_1)}{\pi(H_0)},
\end{align}
where here, as in the remainder of the paper, we use the symbol $\pi$
to denote not just prior, but also posterior distributions on
$\{H_0, H_1\}$.  In the case that we observe $x^n$ for fixed $n$, the
event $A$ is of the form $X^n = x^n$. Plugging this into
\eqref{eq:simple}, the left-hand side becomes the standard definition
of {\em posterior odds}, and the first factor on the right is called
the {\em Bayes factor}.

\subsection{First Sense of Handling Optional Stopping: $\tau$-Independence}
Now, in reality we do not necessarily observe $X^n = x^n$ for fixed
$n$ but rather $X^{\tau} = x^{\tau}$ where $\tau$ is a stopping time
that may itself depend on (past) data (and that in some cases may in
fact be unknown to us).
This stopping time may be defined in terms of a {\em
	stopping rule\/} $f: \bigcup_{i \geq 0}^T \X^i \rightarrow \{ {\tt
	stop}, {\tt continue} \}$. $\tau \equiv \tau(x^T)$ is then defined
as the random variable which, for any sample $x_1, \ldots, x_T$,
outputs the smallest $n$ such that $f(x_1, \ldots, x_n) = {\tt stop}$.
For any given stopping time $\tau$, any $1 \leq n \leq T$ and sequence
of data $x^n = (x_1, \ldots, x_n)$, we say that {\em $x^n$ is
	compatible with $\tau$} if it satisfies $X^n = x^n \Rightarrow \tau
= n$. We let ${\cal X}^{\tau} \subset \bigcup_{i=1}^T {\cal X}^i$ be
the set of all sequences compatible with $\tau$.

Observations take the form $X^{\tau} = x^{\tau}$, which is equivalent to the event \mbox{$X^n=x^n ; \tau = n$}
for some $n$ and some $x^n \in \X^n$ which of necessity must be
compatible with $\tau$.
We can thus instantiate \eqref{eq:simple} to
\begin{align}
\nonumber
\frac{\pi(H_1
	\mid X^n = x^n, \tau = n)}{\pi(H_0 \mid X^n = x^n,\tau = n)}
&= \frac{P(\tau = n
	\mid X^n = x^n, H_1) \cdot \pi(H_1 \mid X^n = x^n)}{P(\tau = n
	\mid X^n = x^n, H_0)\cdot \pi(H_0 \mid X^n = x^n)} = \\
&= \frac{\pi(H_1 \mid X^n = x^n)}{\pi(H_0 \mid X^n = x^n)}.
\label{eq:simplegivenn}
\end{align}
where in the first equality we used Bayes' theorem (keeping $X^n =x^n$ on the right of the conditioning bar throughout); the second equality stems from the fact that $X^n = x^n$ logically implies $\tau =n$, since $x^n$ is  compatible with $\tau$; the probability $P(\tau = n \mid X^n = x^n, H_j)$ must therefore be $1$ for $j = 0,1$.
Combining \eqref{eq:simplegivenn} with Bayes' theorem we get:
\begin{equation}\label{eq:simplefinal}
\overset{\gamma(x^n)}{\overbrace{{\frac{\pi(H_1
				\mid X^n = x^n, \tau = n)}{\pi(H_0 \mid X^n = x^n,\tau = n)}}}} =
\overset{\beta(x^n)}{\overbrace{{\frac{\QM(X^n = x^n)}{\PM(X^n = x^n)}}}}  \cdot \frac{\pi(H_1)}{\pi(H_0)}
\end{equation}
where we introduce the notation $\gamma(x^n)$ for the posterior odds and
$\beta(x^n)$ for the Bayes factor based on sample $x^n$,
calculated as if $n$ were fixed in advance.\footnote{A slightly different way to get to \eqref{eq:simplefinal}, which some may  find even simpler, is to start with $\PM(X^n = x^n, \tau = n) = \PM(X^n = x^n)$ (since $X^n= x^n$ implies $\tau = n$), whence $\pi(H_j \mid X^n = x^n, \tau = n) \propto \bar{P}_j(X^n = x^n,\tau=n) \pi(H_j)= \bar{P}_j(X^n = x^n) \pi(H_j)$.}

We see that the stopping rule plays no role in the expression on the
right. Thus, we have shown that, for any two stopping times $\tau_1$
and $\tau_2$ that are both compatible with some observed $x^n$, the
posterior odds one arrives at will be the same irrespective of whether
$x^n$ came to be observed because $\tau_1$ was used or if $x^n$ came
to be observed because $\tau_2$ was used. We say that the posterior
odds do not depend on the stopping rule $\tau$ and call this property
{\em $\tau$-independence}. Incidentally, this also
justifies that we write the posterior odds as
$\gamma(x^n)$,  a function of $x^n$ alone, without referring to the
stopping time $\tau$.

The fact that the posterior odds given $x^n$ do not depend on the
stopping rule is the first (and simplest) sense in which Bayesian
methods handle optional stopping. It has its roots in the {\em
	stopping rule principle\/}, the general idea that the conclusions
obtained from the data by `reasonable' statistical methods should not
depend on the stopping rule used. This principle was probably first
formulated by Barnard
(\citeyear{barnard1947review,barnard1949statistical});
\cite{barnard1949statistical} very implicitly showed that, under some
conditions, Bayesian methods satisfy the stopping rule principle (and
hence satisfy $\tau$-independence).  Other early sources are
\citet{lindley-1957-statis-parad} and
\citet{edwards-1963-bayes-statis}.  Lindley gave an informal proof in
the context of specific parametric models; in
Section~\ref{sec:generalindependence} we show that, under some
regularity conditions, the result indeed remains true for general
$\sigma$-finite $\PM$ and $\QM$. A special case of our result
(allowing continuous-valued sample spaces but not general measures)
was proven by \cite{RaiffaS61}, and a more general statement about the
connection between the `likelihood principle' and the 'stopping rule
principle' which implies our result in
Section~\ref{sec:generalindependence} can be found in the seminal work
\citep{BergerW88}, who also provide some historical context. Still,
even though not new in itself, we include our result on
$\tau$-independence with general sample spaces and measures since it
is the basic building block of our later results on calibration and
semi-frequentist robustness, which are new.

Finally, we should note that both \cite{RaiffaS61} and
\cite{BergerW88} consider more general stopping rules, which can map to
a probability of stopping instead of just $\{ {\tt stop}, {\tt
	continue} \}$. Also, they allow the stopping rule itself to be
parameterized: one deals with a collection of stopping rules $\{f_{\xi}: \xi \in \Xi\}$ with corresponding stopping times $\{\tau_{\xi}: \xi \in
\Xi\}$, where the parameter $\xi$ is equipped with a prior such that $\xi$ and
$H_j$ are required to be a priori independent. Such
extensions are straightforward to incorporate into our development as
well (very roughly, the second equality in \eqref{eq:simplegivenn} now follows
because, by conditional independence, we must have that $P(\tau_{\xi}
= n \mid X^n = x^n, H_1) = P(\tau_{\xi} = n \mid X^n = x^n, H_0)$); we will not go into such extensions any further in this paper.

\subsection{Second Sense of Handling Optional Stopping: Calibration}
\label{sec:simplecalibration}
An alternative definition of handling optional stopping was introduced
by \cite{rouder-2014-option}. Rouder calls $\gamma(x^n)$ the
\emph{nominal\/} posterior odds calculated from an obtained sample
\(x^n\), and defines the \emph{observed posterior odds} as
\begin{align*}
\frac{\pi(H_1 \mid \gamma(x^n) =c)}{\pi(H_0 \mid \gamma(x^n) =c)}
\end{align*}
as the posterior odds given the nominal odds. Rouder first notes that,
at least if the sample size is fixed in advance to $n$, one expects
these odds to be equal. For instance, if an obtained sample yields
nominal posterior odds of 3-to-1 in favor of the alternative
hypothesis, then it must be 3 times as likely that the sample was
generated by the alternative probability measure. In the terminology
of \cite{heide-2017-why-option}, Bayes is {\em calibrated\/} for a
fixed sample size $n$. Rouder then goes on to note that, if $n$ is
determined by an arbitrary stopping time $\tau$ (based for example on
optional stopping), then the odds will still be equal --- in this
sense, Bayesian testing is well-behaved in the calibration sense
irrespective of the stopping rule/time. Formally, the requirement that the nominal and
observed posterior odds be equal leads us to define the
\emph{calibration hypothesis}, which postulates that $c = \frac{P(H_1
	\mid \gamma=c)}{P(H_0 \mid \gamma=c)}$ holds for any \(c > 0\) that
has non-zero probability. For simplicity, for now we
only consider the case with equal prior odds for $H_0$ and $H_1$ so
that $\gamma(x^n) = \beta(x^n)$. Then  the calibration hypothesis
says that, for arbitrary stopping time $\tau$, for every $c$ such that $\beta(x^{\tau}) = c$ for some $x^{\tau} \in \X^{\tau}$, one has
\begin{align}\label{eq:weakcalibration}
c &=
\frac{P(\beta(x^{\tau}) =c \mid H_1 ) }{P(\beta(x^\tau) =c \mid H_0)}.
\end{align}
In the present simple setting, this hypothesis is easily shown to
hold, because we can write:
\begin{align*}
\frac{P(\beta(X^{\tau}) =c \mid H_1 ) }{P(\beta(X^{\tau}) =c \mid H_0)}
= \frac{\sum_{y \in \X^{\tau}: \beta(y) = c} P(\{y \} \mid H_1) }
{\sum_{y \in \X^\tau;  \beta(y) = c} P(\{y\} \mid H_0) }
= \frac{\sum_{y \in \X^{\tau}:  \beta(y) = c} c P(\{ y \} \mid H_0) }
{\sum_{y \in \X^{\tau}: \beta(y) = c} P(\{ y \} \mid H_0) } = c.
\end{align*}
Rouder noticed that the calibration hypothesis should hold as a
mathematical theorem, without giving an explicit proof; he
demonstrated it by computer simulation in a simple parametric setting. \cite{deng2016continuous}
gave a proof for a somewhat more extended setting yet still with proper priors. In Section~\ref{sec:generalcalibration} we show that a version of the calibration hypothesis continues to hold for general measures based on improper priors, and in Section~\ref{sec:strongcalibration} we extend this further to strong calibration for  group invariance settings as discussed below.

We note that this result, too, relies on the priors
themselves not depending on the stopping time, an assumption which is
violated in several standard default Bayes factor settings. We also
note that, if one thinks of one's priors in a default sense --- they
are practical but not necessarily fully believed --- then the practical
implications of calibration are limited, as shown experimentally by
\cite{heide-2017-why-option}. One would really like a stronger form of calibration in which \eqref{eq:weakcalibration} holds under a whole range of distributions in $H_0$ and $H_1$, rather than in terms of $\PM$ and $\QM$ which average over a prior that perhaps does not reflect one's beliefs fully. For the case that $H_1$ and $H_2$ share a nuisance parameter $g$ taking values in some set $G$, one can define this {\em strong calibration hypothesis\/} as  stating that, for all $c$ with $\beta(x^{\tau}) = c$ for some $x^{\tau} \in \X^{\tau}$, all $g \in G$,
\begin{align}\label{eq:strongcalibration}
c &=
\frac{P(\beta(x^{\tau}) =c \mid H_1, g ) }{P(\beta(x^\tau) =c \mid H_0, g)}.
\end{align}
where $\beta$ is still defined as above; in particular, when
calculating $\beta$ one does not condition on the parameter having the
value $g$, but when assessing its likelihood as in
\eqref{eq:strongcalibration} one does. \cite{heide-2017-why-option}
show that the strong calibration hypothesis certainly does {\em not\/}
hold for general parameters, but they also show by simulations that it
does hold in the practically important case with group invariance and
right Haar priors (Example~\ref{example:A.1} provides an
illustration). In Section~\ref{sec:strongcalibration} we show that in
such cases, one can indeed prove that a version of
\eqref{eq:strongcalibration} holds.

\subsection{Third Sense of Handling Optional Stopping: (Semi-)Frequentist }
In classical, Neyman-Pearson style null hypothesis testing, a main concern is to limit the false positive rate of a
hypothesis test. If this false positive rate is bounded above by some
\(\alpha > 0\), then a null hypothesis significance test (NHST) is said
to have \emph{significance level} \(\alpha\), and if the significance level
is independent of the stopping rule used, we say that the test is
\emph{robust under frequentist optional stopping}.
\begin{definition}\label{def:freqrobust}
	A function  \(S: \bigcup_{i=m}^T \X^i \to \{0,1\}\) is said to be a frequentist sequential test with significance level $\alpha$ and minimal sample size $m$ that is
	\emph{robust under optional stopping relative to $H_0$} if for all $P \in H_0$
	\begin{align*}
	P\left(\exists n, m < n \leq T: S(X^n)=1 \right) \leq \alpha,
	\end{align*}
	i.e.\ the probability that there is an $n$ at which
	$S(X^n)=1$ (`the test rejects $H_0$ when given sample $X^n$') is bounded by $\alpha$.
\end{definition}
In our present setting, we can take $m=0$ (larger $m$ become important in Section~\ref{sec:genfreq}), so $n$ runs from $1$ to $T$ and it is easy to show that, for any $0 \leq \alpha \leq 1$, we have
\begin{equation}\label{eq:markov}
\PM \left( \exists n, 0 < n \leq T: \frac{1}{\beta(x^n)}  \leq \alpha \right) \leq \alpha.
\end{equation}
\begin{proof} For any fixed $\alpha$ and any sequence $x^T = x_1,
	\ldots, x_T$, let $\tau(x^T)$ be the smallest $n$ such that, for the
	initial segment $x^n$ of $x^T$, $\beta(x^n) \geq 1/\alpha$ (if no
	such $n$ exists we set $\tau(x^T) = T$). Then $\tau$ is a stopping
	time, $X^{\tau}$ is a random variable, and the probability in
	\eqref{eq:markov} is equal to the $\PM$-probability that
	$\beta(X^{\tau}) \geq 1/\alpha$, which by Markov's inequality is
	bounded by $\alpha$.
\end{proof}
It follows that, if $H_0$ is a singleton, then the sequential test $S$ that rejects
$H_0$ (outputs $S(X^n) = 1$) whenever $\beta(x^n) \geq 1/\alpha$ is a
frequentist sequential test with significance level $\alpha$ that is
robust under optional stopping.

The fact that Bayes factor testing with singleton $H_0$ handles
optional stopping in this frequentist way was noted by
\citet{edwards-1963-bayes-statis} and also emphasized by
\cite{Good1991}, among many others.  If $H_0$ is not a singleton, then
\eqref{eq:markov} still holds, so the Bayes factor still handles
optional stopping in a mixed frequentist (Type I-error) and Bayesian
(marginalizing over prior within $H_0$) sense. From a frequentist
perspective, one may not consider this to be fully satisfactory, and
hence we call it `semi-frequentist'. In some quite special situations
though, it turns out that the Bayes factor satisfies the stronger
property of being truly robust to optional stopping in the above
frequentist sense, i.e.~\eqref{eq:markov} will hold for all $P\in H_0$
and not just `on average'.
This is illustrated in Example~\ref{example:A.1} below and formalized in
Section~\ref{sec:strongfrequentist}.
%

\section{Discussion: why should one care?}
\label{sec:discussion}
Nowadays, even more so than in the past, statistical tests are often
performed in an on-line setting, in which data keeps coming in
sequentially and one cannot tell in advance at what point the analysis
will be stopped and a decision will be made --- there may indeed be
many such points.  Prime examples include group sequential trials
\citep{proschan2006statistical} and $A/B$-testing, to which all
internet users who visit the sites of the tech giants are subjected.
In such on-line settings, it may or may not be a good idea to use
Bayesian tests. But can and should they be used? Together with the
companion paper \citep{heide-2017-why-option} (DHG from now on), the
present paper sheds some light on this issue. Let us first highlight a central insight from DHG, which
is about the case in which none of the results discussed in the present paper
apply: in many practical situations, many Bayesian statisticians use
priors that are {\em themselves\/} dependent on parts of the data
and/or the sampling plan and stopping time. Examples are Jeffreys
prior with the multinomial model and the Gunel-Dickey default priors
for 2x2 contingency tables advocated by
\cite{jamil-2016-default-gunel}. With such priors, final results
evidently depend on the stopping rule employed, and even though such
methods typically count as `Bayesian', they do not satisfy
$\tau$-independence. The results then become non-interpretable under
optional stopping (i.e.\ stopping using a rule that is not known at the
time the prior is decided upon), and as argued by
\cite{heide-2017-why-option}, the notions of calibration and
frequentist optional stopping even become undefined in such a case.

In such situations, one cannot rely on Bayesian methods to be valid
under optional stopping in any sense at all; in the present paper we
thus focus on the case with priors that are fixed in advance, and that
themselves do not depend on the stopping rule or any other aspects of
the design.
For expository simplicity, we consider the question of whether Bayes factors with such priors are valid under optional stopping in
two extreme settings: in the first setting, the goal of the analysis is purely {\em exploratory\/} --- it should give us some insight in the data and/or suggest novel experiments to gather or novel models to analyze data with. In the second setting we consider the analysis as `final' and the stakes are much higher --- real decisions involving money, health and the like are involved --- a typical example would be a Stage 2 clinical trial, which will decide whether a new medication will be put to market or not.

For the first, {\em exploratory\/} setting, exact error guarantees might
neither be needed at all nor obtainable anyway, so the frequentist
sense of handling optional stopping may not be that important. Yet,
one would still like to use methods that satisfy some basic {\em
	sanity checks\/} for use under optional
stopping. $\tau$-independence is such a check: any method for which it
does not hold is simply not suitable for use in a situation in which
details of the stopping rule may be unknown. Also calibration can be
viewed as such a sanity check: \cite{rouder-2014-option} introduced it
mainly to show that Bayesian posterior odds remain {\em meaningful\/}
under optional stopping: they still satisfy some key property that
they satisfy for fixed sample sizes.

For the second {\em high stakes\/} setting, mere sanity and
interpretability checks are not enough: most researchers would want
more stringent guarantees, for example on Type-I and/or Type-II error
control. At the same time, most researchers would acknowledge that
their priors are far from perfect, chosen to some extent for purposes
of convenience rather than true belief.\footnote{Even De Finetti and
	Savage, fathers of subjective Bayesianism, acknowledged this: see
	Section 5 of DHG.} Such researchers may thus want the
desired Type-I error guarantees to hold for all $P \in H_0$, and not
just in average over the prior as in \eqref{eq:markov}. Similarly, in
the high stakes setting the form of calibration \eqref{eq:weakcalibration}
that can be guaranteed for the Bayes factor would be considered too
weak, and one would hope for a stronger form of calibration as explained
at the end of Section~\ref{sec:simplecalibration}.

DHG show empirically that for some often-used models and priors,
strong calibration can be severely violated under optional
stopping. Similarly, it is possible to show that in general, Type-I
error guarantees based on Bayes factors simply do not hold
simultaneously for all $P \in H_0$ for such models and priors.  Thus,
one should be cautious using Bayesian methods in the high stakes
setting, despite adhortations such as the quote by
\cite{edwards-1963-bayes-statis} in the introduction (or similar
quotes by e.g.~\cite{rouder-2009-bayes}): these existing papers
invariably use $\tau$-independence, calibration or Type-I error
control with simple null hypotheses as a motivation to --- essentially
--- use Bayes factor methods in any situation, including presumably
high-stakes situations and situations with composite null
hypotheses.\footnote{Since the authors of the present papers are
	inclined to think frequentist error guarantees are important, we
	disagree with such claims, as in fact a subset of researchers
	calling themselves Bayesians would as well. To witness, a large
	fraction of recent ISBA (Bayesian) meetings is about frequentist
	properties of Bayesian methods; also the well-known Bayesian authors
	\cite{Good1991} and \cite{edwards-1963-bayes-statis} focus on
	showing that Bayes factor methods achieve a {\em frequentist Type-I
		error\/} guarantee, albeit only for the simple $H_0$ case.}

Still, and this is equally important for practitioners, while
frequentist error control and strong calibration are violated in
general, in some important special cases they do hold, namely if the models
$H_0$ and $H_1$ satisfy a group invariance. We proceed to give an
informal illustration of this fact, deferring the mathematical details
to Section~\ref{sec:strongfrequentist}.
\begin{example}\label{example:A.1}\normalfont
	Consider the one-sample $t$-test as described by
	\cite{rouder-2009-bayes}, going back to
	\cite{jeffreys-1961-theory-of}.
	The test considers normally
	distributed data with unknown standard deviation. The test is meant
	to answer the question whether the data has mean $\mu=0$ (the null
	hypothesis) or some other mean (the alternative hypothesis). Following
	\citep{rouder-2009-bayes}, a Cauchy prior density, denoted by
	$\pi_{\delta}(\delta)$, is placed on the effect size
	$\delta = \mu / \sigma$. The unknown standard deviation is a
	nuisance parameter and is equipped with the improper prior with
	density $\pi_{\sigma}(\sigma) = \frac 1 \sigma$ under both
	hypotheses. This is the so-called {\em right Haar prior\/} for the variance.
	This gives the
	following densities on $n$
	outcomes:
	\begin{align}\label{eq:nulmodel}
	p_{0,\sigma}(x^n)  & =  \frac{1}{(2\pi \sigma^2)^{n/2}} \cdot
	\exp\left(\frac{1}{2\sigma^2} \sum_{i=1}^n x_i^2  \right) \ \ \ [\ = p_{1,\sigma,0}(x^n)\ ]\\
	p_{1,\sigma,\delta}(x^n)  & = \frac{1}{(2\pi \sigma ^2)^{n/2}}
	\cdot
	\exp\left( -\frac{n}{2} \left[ \left(
	\frac{\overline{x}}{\sigma} - \delta \right)^2
	+ \left( \frac{\frac{1}{n}\sum_{i=1}^n (x_i -
		\overline{x})^2}{\sigma^2}
	\right) \right] \right) \text{\ , where } \nonumber \\
	\overline{x} &= \frac{1}{n} \sum_{i=1}^n x_i, \nonumber
	\end{align}
	so that the corresponding Bayesian marginal densities are given by
	\begin{align*}
	\bar{p}_0(x^n)
	&=  \int_0^\infty p_{0,\sigma}(x^n)
	\pi_\sigma(\sigma) \dif \sigma,\\
	\bar{p}_1(x^n)
	&=  \int_0^\infty \int_{-\infty}^\infty p_{1,\sigma,\delta}(x^n)
	\pi_\delta(\delta) \pi_\sigma(\sigma) \dif \delta \dif \sigma =
	\int_0^\infty p_{1,\sigma}(x^n)
	\pi_\sigma(\sigma) \dif \sigma.
	\end{align*}
	Our results in Section~\ref{sec:group} imply that --- under a slight,
	natural restriction on the stopping rules allowed --- the Bayes factor
	$\bar{p}_1(x^n)/\bar{p}_0(x^n)$ is truly robust to optional stopping
	in the above frequentist sense. That is, \eqref{eq:markov} will hold for
	all $P\in H_0$, i.e.\ all $\sigma >0$, and not just `on average'. Thus,
	we can give Type I error guarantees irrespective of the true value of
	$\sigma$. Similarly, strong calibration in the sense of
	Section~\ref{sec:simplecalibration} holds for all $P \in \H_0$.  The
	use of a Cauchy prior is not essential in this construction; the
	result will continue to hold for any proper prior on $\delta$,
	including point priors that put all mass on a single value of
	$\delta$.

	As we show in Section~\ref{sec:group}, these results extend to a
	variety of settings, namely whenever $H_0$ and $H_1$ share a common
	so-called group invariance. In the $t$-test example, it is a scale
	invariance --- effectively this means that for all $\delta$, all $\sigma$, the
	distributions of
	\begin{equation}\label{eq:homegrown}
	\text{ $X_1, \ldots, X_n$ under $p_{1,\sigma,\delta}$, and
		$\sigma X_1, \ldots, \sigma X_n$ under $p_{1,1,\delta}$,
		coincide.}\end{equation}
	For other models, one could have a
	translation invariance; for the full normal family, one has both
	translation and scale invariance; for yet other models, one might have
	a rotation invariance, and so on. Each such invariance is expressed as
	a {\em group\/} --- a set equipped with a binary operation that
	satisfies certain axioms.
	The group corresponding to scale invariance is the set of positive
	reals, and the operator is scalar multiplication or equivalently
	division; similarly, the group corresponding to translation invariance
	is the set of all reals, and the operation is addition.
	
	In the general case, one starts
	with a group $G$ that satisfies certain further restrictions (detailed in Section~\ref{sec:group}), a model $\{ p_{1,g,\theta} : g \in G, \theta \in
	\Theta\}$ where $g$ represents the invariant parameter (vector)
	and the parameterization must be such that the analogue of
	\eqref{eq:homegrown} holds.  In the example above $g = \sigma$ is
	the variance and $\theta$ is set to $\delta \coloneqq \mu/\sigma$.  One then
	singles out a special value of $\theta$, say $\theta_0$, one sets $H_0
	\coloneqq \{ p_{1,g,\theta_0}: g\in G \}$; within $H_1$ one
	puts an arbitrary prior on $\theta$. For every group invariance, there
	exists a corresponding {\em right Haar prior\/} on $G$; one
	equips both models with this prior on
	$G$. Theorem~\ref{thm:strong-calibration-under-os}
	and~\ref{th:frequentist-os-with-nuisance-pars} imply that in all
	models constructed this way, we have strong calibration and Type-I
	error control uniformly for all $g \in G$.  While this is
	hinted at in several papers
	(e.g.~\citep{bayarri2016rejection,dass-2003-unified-condit}) and the
	special case for the Bayesian $t$-test was implicitly proven in
	earlier work by \cite{lai1976confidence}, it seems to never have been
	proven formally in general before.
\end{example}
Our results thus imply that in some situations (group invariance) with
composite null hypotheses, Type-I error control for all $P \in H_0$ under
optional stopping is possible with Bayes factors. What about Type-II
error control and composite null hypotheses that do {\em not\/} satisfy a group
structure? This is partially addressed by the {\em safe testing\/}
approach of \cite{safetesting} (see also \cite{howard2018uniform} for
a related approach). They show that for completely arbitrary $H_0$ and
$H_1$, for any given prior $\pi_1$ on $H_1$, there exists a
corresponding prior $\pi_0$ on $H_0$, the {\em reverse information
	projection prior}, so that, for all $P \in H_0$, one has Type-I
error guarantees under frequentist {\em optional continuation}, a
weakening of the idea of optional stopping. Further, if one wants to
get control of Type-II error guarantees under optional
stopping/continuation, one can do so by first choosing another special
prior $\pi^*_1$ on $H_1$ and picking the corresponding $\pi^*_0$ on
$H_0$. Essentially, like in `default' or `objective' Bayes approaches,
one chooses special priors in lieu of a subjective choice; but the
priors one ends up with are sometimes quite different from the
standard default priors, and, unlike these, allow for frequentist
error control under optional stopping.
\section{The General Case}
\label{sec:general}
Let $(\Omega, \mathcal{F})$ be a measurable space. Fix some $m \geq 0$
and consider a sequence of functions $X_{m+1}, X_{m+2}, \ldots$ on
$\Omega$ so that each $X_n$, $n > m$ takes values in some fixed set
(`outcome space') ${\cal X}$ with associated $\sigma$-algebra
$\Sigma$. When working with proper priors we invariably take $m=0$ and
then we define $X^n \coloneqq (X_{1}, X_{2}, \ldots, X_n)$ and we let
$\Sigma^{(n)}$ be the $n$-fold product algebra of $\Sigma$. When
working with improper priors it turns out to be useful (more explanation further below) to take $m > 0$
and define an {\em initial sample\/} random variable $\initialX$ on
$\Omega$, taking values in some set $\initialspace \subseteq \cX^m$
with associated $\sigma$-algebra $\initialsigma$. In that case we set,
for $n \geq m$, $\samplespacen = \{x^n = (x_1, \ldots, x_n) \in \cX^n:
x^m = (x_1,\ldots, x_m) \in \initialspace\}$, and $X^n \coloneqq (\initialX,
X_{m+1}, X_{m+2}, \ldots, X_n)$ and we let $\Sigma^{(n)}$ be
$\initialsigma \times \prod_{j={m+1}}^n \Sigma$.  In either case, we
let $\F_n$ be the $\sigma$-algebra (relative to $\Omega$) generated by
$(X^n,\Sigma^{(n)})$. Then $(\F_n)_{n=m, m+1, \ldots}$ is a filtration
relative to $\F$ and if we equip $(\Omega,\F)$ with a distribution $P$
then $\initialX, X_{m+1}, X_{m+2}, \ldots$ becomes a random process
adapted to $\F$.  A {\em stopping time\/} is now generalized to be a
function $\tau: \Omega \to \{m+1, m+2, \ldots \} \cup \{\infty\}$ such
that for each $n > m$, the event $\{ \tau = n \}$ is
$\F_n$-measurable; note that we only consider stopping after $m$ initial outcomes.  Again, for a given stopping time $\tau$ and
sequence of data $x^n = (x_1, \ldots, x_n)$, we say that {\em $x^n$ is
	compatible with $\tau$} if it satisfies $X^n = x^n \Rightarrow \tau
= n$, i.e.\ $\{ \omega \in \Omega \mid X^n(\omega) = x^n \} \subset
\{\omega \in \Omega \mid \tau(\omega) = n \}$.

$H_0$ and $H_1$ are now sets of probability distributions on
$(\Omega, \mathcal{F})$. Again one
writes $H_j = \{P_{\theta \mid j} \mid \theta \in \Theta_j\}$ where now
the parameter sets $\Theta_j$ (which, however, could itself be
infinite-dimensional) are themselves equipped with suitable
$\sigma$-algebras.

We will still represent both $H_0$ and $H_1$ by unique measures
\(\PM\) and \(\QM\) respectively, which we now allow to be based on
\eqref{eq:bayesfactor} with improper priors $\pi_0$ and $\pi_1$ that
may be infinite measures. As a result $\PM$ and $\QM$ are positive real measures
that may themselves be infinite. We also allow ${\cal X}$ to be a
general (in particular uncountable) set. Both non-integrability and
uncountability cause complications, but these can be overcome if
suitable Radon-Nikodym derivatives exist. To ensure this,
we will assume that for all $n \geq\max\{ m, 1\}$,
for all $k \in \{0,1\}$ and $\theta \in \Theta_k$,
$P^{(n)}_{\theta|k}$, $\PM^{(n)}$ and $\QM^{(n)}$ are all mutually
absolutely continuous and that
the measures $\QM^{(n)}$ and $\PM^{(n)}$ are
$\sigma$-finite.
%
Then there also exists a measure $\rho$ on $(\Omega,\F)$ such that, for all such $n$, $\QM^{(n)}$, $\PM^{(n)}$ and $\rho^{(n)}$ are all mutually absolutely continuous: we can simply take $\rho^{(n)}= \PM^{(n)}$, but
in practice, it is often possible
and convenient to take $\rho$ such that $\rho^{(n)}$ is the Lebesgue
measure on ${\mathbb R}^n$, which is why we explicitly introduce $\rho$ here.

The absolute continuity  conditions guarantee that all required Radon-Nikodym derivatives exist. Finally, we assume that the
posteriors $\pi_k(\Theta_k \mid x^m)$
(as defined in the standard manner in \eqref{eq:posteriortheta}
below; when $m=0$ these are just the priors) are proper probability measures (i.e.\ they integrate to 1) for all $x^m \in \initialspace$.
This final requirement is the reason why we sometimes need to
consider $m>0$ and nonstandard sample spaces $\samplespacen$ in the
first place: in practice , one usually starts with the standard
setting of a $(\Omega,\F)$ where $m=0$ and all $X_i$ have the same
status.  In all practical situations with improper priors $\pi_0$
and/or $\pi_1$ that we know of, there is a smallest finite $j$ and a
set ${\cal X}^{\circ} \subset \cX^j$ that has measure $0$ under all
probability distributions in $H_0 \cup H_1$, such that, restricted
to the sample space $\cX^j \setminus {\cal X}^{\circ}$, the measures
$\QM^{(j)}$ and $\PM^{(j)}$ are $\sigma$-finite and mutually
absolutely continuous, and the posteriors $\pi_k(\Theta_k \mid x^j)$
are proper probability measures. One then sets $m$ to equal
this $j$, and sets
$\initialspace \coloneqq \cX^m \setminus {\cal X}^{\circ}$, and the required
properness will be guaranteed.  Our initial sample $\initialX$ is a
variation of what is called (for example, by
\citet{bayarri2012criteria}) a {\em minimal sample}. Yet, the sample size
of a standard minimal sample is itself a random quantity; by
restricting $\cX^m$ to $\initialspace$, we can take its sample size
$m$ to be constant rather than random, which will greatly simplify
the treatment of optional stopping with group invariance; see
Example~\ref{example:A.1} and~\ref{example:B.1} below.

We henceforth refer to the setting now defined (with $m$ and initial
space $\initialspace$ satisfying the requirements above) as the {\em
	general case}.

We need an analogue of \eqref{eq:simplefinal} for this general case.
If $\PM$ and $\QM$ are probability measures, then there is still a
standard definition of conditional probability distributions
$P(H \mid \A)$ in terms of conditional expectation for any given
$\sigma$-algebra $\A$; based on this, we can derive the required
analogue in two steps. First, we consider the case that
$\tau \equiv n$ for some $n > m$. We know in advance that we observe $X^n$ for a fixed
$n$: the appropriate $\A$ is then $\F_n$, $\pi(H \mid \A)(\omega)$ is
determined by $X^n(\omega)$ hence can be written as $\pi(H \mid X^n)$,
and a straightforward calculation gives
that
\begin{equation}\label{eq:fixedn}
\frac{\pi(H_1 \mid X^n = x^n)}{\pi(H_0 \mid X^n = x^n)}
= \left( \left( \frac{d\QM^{(n)}/d
	\rho^{(n)}}{d\PM^{(n)}/d \rho^{(n)}}\right) (x^n) \right)
\cdot \frac{\pi(H_1)}{\pi(H_0)}
\end{equation}
where $(d \QM^{(n)}/d\rho^{(n)})$ and $(d \PM^{(n)}/d\rho^{(n)})$ are
versions of the Radon-Nikodym derivatives defined relative to
$\rho^{(n)}$.
The second step is now to follow exactly the same steps as in the
derivation of \eqref{eq:simplefinal}, replacing $\beta(X^n)$ by
\eqref{eq:fixedn} wherever appropriate (we omit the details). This
yields, for any $n$ such that \mbox{$\rho(\tau = n) > 0$}, and for
$\rho^{(n)}$-almost every  $x^n$ that is compatible with $\tau$,
\begin{equation}  \label{eq:posterior-odds}
\overset{\gamma_n}{\overbrace{\frac{\pi(H_1
			\mid  x^n)}{\pi(H_0 \mid  x^n)}}} =
\frac{\pi(H_1
	\mid X^n = x^n, \tau = n)}{\pi(H_0 \mid X^n = x^n,\tau = n)} =
\overset{\beta_n}{\overbrace{\left( \left(
		\frac{d\QM^{(n)}/d \rho^{(n)}}{d\PM^{(n)}/d \rho^{(n)}}\right) (x^n) \right)}}
\cdot  \frac{\pi(H_1)}{\pi(H_0)},
\end{equation}
where here, as below, for $n \geq m$,
we abbreviate $\pi(H_k \mid X^n = x^n)$ to $\pi(H_k \mid x^n)$.

The above expression for the posterior is valid if $\PM$ and $\QM$ are
probability measures; we will simply take it as the {\em definition\/}
of the Bayes factor for the general case. Again this coincides with standard usage for the improper prior case. In particular, let us  define the conditional posteriors and  Bayes factors given $\initialX = x^m$ in the standard manner, by the formal application of Bayes'  rule, for $k= 0,1$ and measurable $\Theta'_k \subset \Theta_k$ and $\F$-measurable $A$,
\begin{align}\label{eq:posteriortheta}
\pi_k(\Theta'_k \mid x^m) & \coloneqq   \frac{\int_{\Theta'_k} \frac{d P_{\theta|k}^{(m)}}{d \rho^{(m)}}  (x^m)
	d \pi_k(\theta)}{\int_{\Theta_k} \frac{d P_{\theta|k}^{(m)}}{d \rho^{(m)}}  (x^m)  d\pi_k(\theta)} \\ \label{eq:conditionalmarginal}
\KM (A \mid  x^m) & \coloneqq \KM (A \mid \initialX = x^m) \coloneqq
\int_{\Theta_k} P_{\theta|k}(A \mid \initialX = x^m) \dif \pi_k(\theta \mid x^m),
\end{align}
where $P_{\theta|k}(A \mid \initialX=x^m)$ is defined as the value that (a
version of) the conditional probability $P_{\theta|k}(A \mid \F_m)$
takes when $\initialX = x^m$, and is thus defined up to a set of
$\rho^{(m)}$-measure 0.

With these definitions, it is straightforward to derive the following
{\em coherence property}, which automatically holds if the priors are
proper, and which in combination with \eqref{eq:posterior-odds}
expresses that first updating on $x^m$ and then on
$x_{m+1}, \ldots, x_n$ (multiplying posterior odds given $x^m$ with the Bayes factor for $n$ outcomes given $X^m =x^m$, which we denote by $\beta_{n|m}$) has the same result as updating based on the
full $x_1, \ldots, x_n$ at once (i.e.\ multiplying the prior odds with the unconditional Bayes factor $\beta_n$ for $n$ outcomes):
\begin{align}  \label{eq:coherence}
\frac{\pi(H_1
	\mid X^n = x^n, \tau = n)}{\pi(H_0 \mid X^n = x^n,\tau = n)} & =
\overset{\beta_{n|m}}{\overbrace{\left(
		\frac{d\QM^{(n)}(\cdot \mid x^m)}{d\PM^{(n)}(\cdot \mid x^m)} (x^n) \right)}}
\cdot  \frac{\pi(H_1 \mid x^m)}{\pi(H_0 \mid x^m)}.
\end{align}

\subsection{$\tau$-independence, general case}
\label{sec:generalindependence}
The general version of the claim that the posterior odds do not depend
on the specific stopping rule that was used is now immediate, since
the expression \eqref{eq:posterior-odds} for the Bayes factor does not
depend on the stopping time $\tau$.
\subsection{Calibration, general case}
\label{sec:generalcalibration}
We will now show that the calibration hypothesis continues to hold in
our general setting. From here onward, we make the further reasonable
assumption that for every $x^m \in \initialspace$, $\PM(\tau = \infty
\mid x^m) = \QM(\tau = \infty \mid x^m) = 0$ (the stopping time is
almost surely finite), and we define ${\cal T}_{\tau} \coloneqq \{n \in
{\mathbb N}_{>m}\mid \PM(\tau=n) > 0\}$.

To prepare further, let
$\{B_{j}\mid j \in {\cal T}_{\tau} \}$ be any
collection of positive random variables such that for each $j \in
{\cal T}_{\tau}$, $B_j$ is $\F_j$-measurable. We can define the {\em
	stopped\/} random variable $B_{\tau}$ as
\begin{align}
\label{eq:def-Bt}
B_{\tau} &\coloneqq \sum_{j=0}^{\infty} \ind[\tau=j]  B_j
= \sum_{j=m+1}^{\infty} \ind[\tau=j] B_j,
\end{align}
where we note that, under this definition, $B_{\tau}$ is well-defined
even if ${\mathbf E}_{\PM}[\tau] = \infty$.

We can  define the
induced measures on the positive real line under the null and
alternative hypothesis
for any probability measure $P$ on $(\Omega, \F)$:
\begin{align}\label{eq:induced}
\mugeneral{P}{B}{\tau}
&: \B(\R_{>0}) \to [0,1]: A \mapsto P \left( B_{\tau}^{-1}(A) \right).
\end{align}
%
where \(\mathcal{B}(\mathbb{R}_{>0})\) denotes the Borel
\(\sigma\)-algebra of \(\mathbb{R}_{>0}\). Note that, when we refer to
$\mugeneral{P}{B}{n}$, this is identical to $\mugeneral{P}{B}{\tau}$
for the stopping time
$\tau$ which on all of $\Omega$ stops at $n$.
The following lemma is crucial for passing from fixed-sample size to stopping-rule based results.
\begin{lemma}\label{lem:caligali}
	Let
	${\cal T}_{\tau}$ and  $\{B_{n}\mid n \in {\cal T}_{\tau} \}$  be as above.
	Consider two probability measures $P_{0}$ and $P_{1}$ on $(\Omega, \F)$. Suppose that for all $n  \in {\cal T}_{\tau}$, the following {\em fixed-sample size calibration property\/} holds:
	\begin{align}
	\label{eq:ratio-OS-minibatch}
	\text{for some fixed $c > 0$, \ }
	\mugeneral{P_0}{B}{n}
	\text{-almost all\ } b:  \frac{P_1(\tau = n)}{P_0(\tau = n)} \cdot \rnd{\mugencond{P_1}{B}{n}{\tau = n}}{\mugencond{P_0}{B}{n}{\tau = n}} (b) & = c \cdot b.
	\end{align}
	Then
	we have
	\begin{align}
	\label{eq:ratio-OS-minibatchb}
	\text{for $\mugeneral{P_0}{B}{\tau}$-almost all \(b\)\ :\ }
	\rnd{\mugeneral{P_1}{B}{\tau}}{\mugeneral{P_0}{B}{\tau}}(b) &= c \cdot b.
	\end{align}
\end{lemma}
The proof is in Section~\ref{app:proofs} in the supplementary material.

In this subsection we apply this lemma to  the measures $\KM(\cdot \mid x^m)$ for arbitrary fixed $x^m \in \initialspace$, with their induced measures  $\mumargcond{0}{\gamma}{\tau}{x^m},
\mumargcond{1}{\gamma}{\tau}{x^m}$ for  the {\em stopped posterior odds\/} $\gamma_{\tau}$. Formally, the posterior odds $\gamma_n$ as
defined in \eqref{eq:posterior-odds} constitute a random variable for each
$n$, and, under our mutual absolute continuity assumption for $\PM$
and $\QM$, $\gamma_n$ can be directly written as \(\rnd{\Qn}{\Pn}
\cdot \pi(H_1) / \pi(H_0)\). Since, by definition, the measures $\KM(\cdot \mid x^m)$ are probability measures, the Radon-Nikodym derivatives in \eqref{eq:ratio-OS-minibatch} and \eqref{eq:ratio-OS-minibatchb} are well-defined.
\begin{lemma}\label{lem:calibali}
	We have for all $x^m \in \initialspace$, all $n > m$:
	\begin{align}
	\label{eq:ratio-OS-mini}
	\text{for \( \mumargcond{0}{\gamma}{n}{x^m}\)-almost all \(b  \)\ :\ }  \frac{\mumarginal{1}{\gamma}{n}(\tau = n \mid x^m)}{\mumarginal{0}{\gamma}{n}(\tau = n \mid x^m)} \cdot
	\rnd{
		\mumargcond{1}{\gamma}{n}{x^m}}{
		\mumargcond{0}{\gamma}{n}{x^m}} (b) &= \frac{\pi(H_0 \mid x^m)}{\pi(H_1 \mid x^m)} \cdot b.
	\end{align}
\end{lemma}
Combining the two lemmas now immediately gives \eqref{eq:ratio-OS} below, and combining further with \eqref{eq:coherence} and \eqref{eq:posterior-odds} gives \eqref{eq:ratio-OSb}:
\begin{corollary}
	\label{cor:calibration-under-OS}
	In the setting considered above, we have for all $x^m \in \initialspace$:
	\begin{align}
	\label{eq:ratio-OS}
	\text{for $\mumargcond{0}{\gamma}{\tau}{x^m}$-almost all $b$\ :\ }
	\frac{\pi(H_1 \mid x^m)}{\pi(H_0 \mid x^m)} \cdot \rnd{\mumargcond{1}{\gamma}{\tau}{x^m}}{\mumargcond{0}{\gamma}{\tau}{x^m}}(b) &=
	b,
	\end{align}
	and also
	\begin{align}
	\label{eq:ratio-OSb}
	\text{for $\mumargcond{0}{\gamma}{\tau}{x^m}$-almost all $b$\ :\ }
	\frac{\pi(H_1)}{\pi(H_0)} \cdot \rnd{\mumarginal{1}{\gamma}{\tau}}{\mumarginal{0}{\gamma}{\tau}}(b) &=
	b,
	\end{align}
	In words, the posterior odds remain calibrated under any stopping rule
	$\tau$ which stops almost surely at times $m < \tau < \infty$.
\end{corollary}
For discrete and strictly positive measures with prior odds
$\pi(H_1)/\pi(H_0) = 1$, we always have $m=0$, and \eqref{eq:ratio-OS}
is equivalent to \eqref{eq:weakcalibration}. Note that
$\mumargcond{0}{\gamma}{\tau}{x^m}$-almost everywhere in
\eqref{eq:ratio-OS} is equivalent to
$\mumargcond{1}{\gamma}{\tau}{x^m}$-almost everywhere
because the two measures are
assumed to be mutually absolutely continuous.

\subsection{(Semi-)Frequentist Optional Stopping}
\label{sec:genfreq}
In this section we consider our general setting as in the beginning of
Section~\ref{sec:generalcalibration}, i.e.\ with the added assumption
that the stopping time is a.s.\ finite, and with ${\cal T}_{\tau} \coloneqq
\{j \in {\mathbb N}_{>m}\mid \PM(\tau=j) > 0\}$.

Consider any
initial sample $x^m \in \initialspace$ and let $\PM \mid x^m$ and $\QM
\mid x^m$ be the conditional Bayes marginal distributions as defined
in \eqref{eq:conditionalmarginal}.
We first note that, by Markov's inequality, for any nonnegative random variable $Z$ on $\Omega$ with, for all $x^m \in \initialspace$, ${\mathbf E}_{\PM \mid x^m}[Z] \leq 1$, we must have, for $0 \leq \alpha \leq 1$, $\PM(Z^{-1} \leq \alpha \mid x^m) \leq
{\bf E}_{\PM \mid x^m}[ Z]/\alpha^{-1} \leq \alpha$.
\begin{proposition}\label{prop:safetest}
	Let $\tau$ be any stopping rule satisfying our requirements. Let
	$\beta_{\tau|m}$ be the stopped Bayes factor given $x^m$, i.e., in
	accordance with \eqref{eq:def-Bt},
	$\beta_{\tau|m} = \sum_{j=m+1}^{\infty} \ind[\tau=j] \beta_{j|m}$
	with $\beta_{j|m}$ as given by \eqref{eq:coherence}. Then
	$\beta_{\tau|m}$ satisfies, for all $x^m \in \initialspace$,
	${\bf E}_{\PM \mid x^m}[ \beta_{\tau|m}] \leq 1$, so that, by the
	reasoning above,
	$\PM(\frac1{\beta_{\tau|m}} \leq \alpha \mid x^m) \leq \alpha$.
\end{proposition}
\begin{proof}
	We have
	\begin{multline}
	{\bf E}_{\PM \mid x^m}\left[ \gamma_{\tau} \right] \
	= \int b  \PM^{[\gamma_{\tau}]} (\dif b \mid x^m)
	= \nonumber \\
	\int   \frac{\dif \QM^{[\gamma_{\tau}]}(b \mid x^m)}{\dif
		\PM^{[\gamma_{\tau}]} (b \mid x^m)} \cdot \frac{\pi(H_1 \mid
		x^m)}{\pi(H_0 \mid x^m)} \;  \PM^{[\gamma_{\tau}]} (\dif b \mid x^m)
	= \frac{\pi(H_1 \mid x^m)}{\pi(H_0 \mid x^m)}, \nonumber
	\end{multline}
	where the first equality follows by definition of expectation, the
	second follows from Corollary~\ref{cor:calibration-under-OS}, and the
	third follows from the fact that the integral equals $1$.
	
	But now note that
	$$\beta_{\tau|m} = \sum_{j=m+1}^{\infty} \ind[\tau=j] \beta_{j|m}
	= \sum_{j=m+1}^{\infty} \ind[\tau=j] \gamma_{j} \cdot \frac{\pi(H_0 \mid x^m)}{\pi(H_1 \mid x^m)} = \gamma_{\tau} \cdot \frac{\pi(H_0 \mid x^m)}{\pi(H_1 \mid x^m)},
	$$
	where the second equality follows from \eqref{eq:coherence} together with the first equality in \eqref{eq:posterior-odds}.
	Combining the two equations we get:
	$$
	{\bf E}_{\PM \mid x^m}\left[ \beta_{\tau|m}\right] =  {\bf E}_{\PM \mid x^m}\left[ \gamma_{\tau} \cdot \frac{\pi(H_0 \mid x^m)}{\pi(H_1 \mid x^m)}  \right] = 1.
	$$
\end{proof}
The desired result now follows by plugging in a particular
stopping rule: let \(S: \bigcup_{i=m+1}^{\infty} \X^i \to \{0,1\}\) be the frequentist sequential test  defined by setting, for all $n> m$, $x^n \in \samplespacen$:  $S(x^n) = 1$ if and only if $\beta_{n|m} \geq 1/\alpha$.
\begin{corollary}\label{cor:thailand}
	Let $t^* \in \{m+1, m+2, \ldots \} \cup \{ \infty \}$ be the smallest $t^* > m$ for which $\beta_{t|m}^{-1} \leq \alpha$. Then for  arbitrarily large $T$,
	when applied to the stopping rule $\tau \coloneqq \min \{T, t^*\}$, we find that
	$$\PM(\exists n, m < n \leq T: S(X^n) = 1 \mid x^m)
	=\PM(\exists n, m < n \leq T: \beta^{-1}_{n|m} \leq \alpha \mid x^m)
	\leq \alpha.$$
\end{corollary}
The corollary implies that the test $S$ is robust under optional stopping in the frequentist sense relative to $H_0$  (Definition~\ref{def:freqrobust}). Note that, just  as in the simple case, the setting is really just `semi-frequentist' whenever $H_0$ is not a singleton.
\section{Optional stopping with group invariance}
\label{sec:group}
Whenever the null hypothesis is composite, the previous results only
hold under the marginal distribution $\PM$ or, in the case of improper
priors, under $\PM (\cdot \mid X^m = x^m)$. When a group structure can
be imposed on the outcome space and (a subset of the) parameters that
is joint to $H_0$ and $H_1$, stronger results can be derived for
calibration and frequentist optional stopping. Invariably, such
parameters function as {\em nuisance parameters\/} and our results are
obtained if we equip them with the so-called {\em right Haar
	prior\/} which is usually improper. Below we show how
we then obtain results that simultaneously hold for {\em all\/} values of
the nuisance parameters. Such cases include many standard testing
scenarios such as the (Bayesian variations of the) $t$-test, as
illustrated in the examples below.  Note though that our results do
not apply to settings with improper priors for which no group
structure exists. For example, if $P_{\theta |0}$ expresses that
$X_1, X_2, \ldots$ are i.i.d.\ Poisson$(\theta)$, then from an
objective Bayes or MDL point of view it makes sense to adopt Jeffreys'
prior for the Poisson model; this prior is improper, allows initial
sample size $m=1$, but does not allow for a group structure. For such
a prior we can only use the marginal results
Corollary~\ref{cor:calibration-under-OS} and
Corollary~\ref{cor:thailand}. Group theoretic preliminaries, such as definitions of a (topological) group, the right Haar measure, et cetera can be found in Section~\ref{app:groupprel} of the supplementary material.
\subsection{Background for fixed sample sizes}
\label{sec:general-background}
Here we prepare for our  results by  providing some
general background on invariant priors for Bayes factors with fixed
sample size $n$ on models with nuisance parameters that admit a group
structure, introducing the right Haar measure, the corresponding Bayes marginals, and (maximal) invariants. We use these results in
Section~\ref{sec:strongfixed} to derive Lemma~\ref{lem:strongfixedn}, which  gives us
a strong version of calibration for fixed
$n$. The setting is extended to variable stopping times in
Section~\ref{sec:extending}, and then Lemma~\ref{lem:strongfixedn} is
used in this extended setting to obtain our strong optional stopping
results in Section~\ref{sec:strongcalibration}
and~\ref{sec:strongfrequentist}.

For now, we assume a sample space $\samplespacen$ that is locally compact and
Hausdorff, and that is a subset of some product space $\cX^n$ where $\cX$ is
itself locally compact and Hausdorff.  This requirement is met, for
example, when $\mathcal{X} = \mathbb{R}$ and $\samplespacen = \cX^n$. In practice,
the space $\samplespacen$ is invariably a subset of $\cX^n$ where
some null-set is removed for technical reasons that will become apparent below.
We associate $\samplespacen$ with its Borel
$\sigma$-algebra which we denote as $\F_n$. Observations are denoted
by the random vector $X^n = (X_1, \ldots, X_n) \in \samplespacen$.  We thus
consider outcomes of fixed sample size, denoting these as
$x^n \in \samplespacen$, returning to the case with stopping times in
Section~\ref{sec:strongcalibration} and~\ref{sec:strongfrequentist}.

From now on we let $G$ be a locally compact group $G$ that acts topologically and
properly\footnote{A group acts properly on a set $Y$ if the mapping
	$\psi: Y \times G \mapsto Y \times Y$ defined by
	$\psi(y, g) = (y \cdot g, y)$ is a proper mapping, i.e.\ the inverse
	image of $\psi$ of each compact set in $Y \times Y$ is a compact set
	in $Y \times G$. (\citet{eaton-1989-group-invar}, Definition~5.1)}
on the right of $\samplespacen$.
As hinted to before, this proper action requirement sometimes forces
the removal from $\cX^n$ of some trivial set with measure zero under
all hypotheses involved. This is demonstrated at the end of Example
\ref{example:A.1} below.

Let  $\genericP{0}$ and $\genericP{1}$ (notation to become clear below) be two arbitrary
probability distributions on $\samplespacen$ that are mutually absolutely continuous.
We will now generate hypothesis classes $H_0$ and $H_1$, both sets of
distributions on $\samplespacen$ with parameter space $G$, starting
from $\genericP{0}$ and $\genericP{1}$, where $e \in G$ is the group
identity element.
The group action
of $G$ on $\samplespacen$ induces a group action on these measures
defined by
\begin{align}\label{eq:definitiePkg}
P_{k,g}(A) \coloneqq (\genericP{k} \cdot g)(A)
\coloneqq \genericP{k}(A \cdot g^{-1})
= \int \ind[A](x \cdot g) \, \genericP{k}(\dif x)
\end{align}
for any set $A \in \F_n$, $k=0, 1$. When applied to $A = \samplespacen$, we get $P_{k,g}(A) = 1$, for all $g \in G$, whence
we have created two sets of probability measures parameterized by $g$, i.e.,
\begin{align}\label{eq:standardform}
H_0 \coloneqq \{ P_{0,g} \mid g \in G \}
\ \ ;\ \
H_1 \coloneqq \{ P_{1,g} \mid g \in G \}.
\end{align}
In this context, $g \in G$, can typically be viewed as nuisance
parameter, i.e.\ a parameter that is not directly of interest, but
needs to be accounted for in the analysis. This is illustrated in Example~\ref{example:A.1} and Example~\ref{example:B.1} below. The examples also illustrate how to extend this setting  to cases where
there are more parameters than just $g \in G$ in either $H_0$ or $H_1$. We  extend the whole
setup to our general setting with non-fixed $n$ in
Section~\ref{sec:strongcalibration}.

%
%
%
We use the right Haar measure for $G$ as a prior to define the Bayes marginals:
\begin{align}\label{eq:marginalsmetg}
\KM(A) = \int_G \int_{\samplespacen} \ind[A] \dif P_{k,g} \, \nu(\dif g)
\end{align}
for $k= 0,1$ and $A \in \F_n$.
Typically, the right Haar measure is improper so that the
Bayes marginals $\KM$ are not integrable.
Yet, in all cases of interest, they are (a) still $\sigma$-finite,
and, (b), $\PM$, $\QM$ and all distributions $P_{k,g}$ with
$k=0,1$ and $g \in G$ are mutually absolutely continuous; we
will henceforth assume that (a) and (b) are the case.

\paragraph{Example~\ref{example:A.1} (continued)}\label{example:A.1a}
Consider the $t$-test of Example~\ref{example:A.1}. For consistency with
the earlier Example~\ref{example:A.1}, we abbreviate for general
measures $P$ on $\samplespacen$, $(\dif P /\dif \lambda)$ (the density
of distribution $P$ relative to Lebesgue measure on ${\mathbb R}^n$)
to $p$.  Normally, the one-sample $t$-test is viewed as a test between
$H_0 = \{P_{0,\sigma} \mid \sigma \in {\mathbb R}_{> 0} \}$ and
$H'_1 = \{ P_{1,\sigma,\delta} \mid \sigma \in {\mathbb R}_{> 0},
\delta \in {\mathbb R} \}$, but we can obviously also view it as test
between $H_0$ and $H_1 = \{P_{1,\sigma} \}$ by integrating out the
parameter $\delta$ to obtain
\begin{align}
\label{eq:example-integrating-out}
p_{1,\sigma}(x^n)  & = \int p_{1,\sigma,\delta}(x^n)   \pi_\delta(\delta)  \dif \delta.
\end{align}
The nuisance parameter $\sigma$ can be identified with the group of scale
transformations \mbox{$G = \{ c \mid c \in \mathbb{R}_{>0} \}$}. We
thus let the sample space be $\samplespacen = \mathbb{R}^n \setminus
\{0\}^n$, i.e., we remove the measure-zero set $\{0\}^n$, such that
the group action is proper on the sample space. The group action is
defined by $x^n \cdot c = c \, x^n$ for $x^n \in \samplespacen, c \in
G$.  Take $e = 1$ and let, for $k=0,1$, $P_{k,e}$ be the distribution
with density $p_{k,1}$ as defined in \eqref{eq:nulmodel} and
\eqref{eq:example-integrating-out}. The measures $P_{0,g}$ and
$P_{1,g}$ defined by \eqref{eq:definitiePkg} then turn out to have the
densities $p_{0,\sigma}$ and $p_{1,\sigma}$ as defined above, with
$\sigma$ replaced by $g$. Thus, $H_0$ and $H_1$ as defined by
\eqref{eq:nulmodel} and \eqref{eq:example-integrating-out} are indeed
in the form \eqref{eq:standardform} needed to state our results.\\ \ \\
In most standard invariant settings, $H_0$ and $H_1$ share the same vector of nuisance parameters, and one can reduce $H_0$ and $H_1$ to
\eqref{eq:standardform} in the same way as above, by integrating out all other
parameters; in the example above, the only non-nuisance parameter was $\delta$.
The scenario of Example~\ref{example:A.1} can be generalized to a surprisingly wide variety of statistical models. In practice we often start with a model $H_1 = \{ P_{1,\gamma,\delta} : \gamma \in \Gamma, \theta \in \Theta\}$ that implicitly already contains a group structure, and we single out a special subset $\{P_1,\gamma,\theta_0: \gamma \in \Gamma \}$; this is what we informally described in Example~\ref{example:A.1}. More generally,  we can start with potentially
large (or even nonparametric) hypotheses
\begin{align}
\label{eq:there}
{H'_k = \{ P_{\theta'|k} : \theta' \in \Theta'_k\}}
\end{align}
which at first are not related to any group invariance, but which we
want to equip with an additional nuisance parameter determined by a
group $G$ acting on the data. We can turn this into an instance of the
present setting by first choosing,for $k=0,1$, a proper prior density
$\pi_k$ on $ \Theta'_k$, and defining $P_{k,e}$ to equal the
corresponding Bayes marginal, i.e.
\begin{equation}\label{eq:almost}
P_{k,e}(A) \coloneqq \int P_{\theta' \mid k}(A) \, \dif \pi_k(\theta').
\end{equation}
We can then generate $H_k= {\{P_{k,g} \mid g \in G\}}$ as in
\eqref{eq:definitiePkg} and \eqref{eq:standardform}. In the example
above, $H'_1$ would be the set of all Gaussians with a single fixed
variance $\sigma_0^2$ and $\Theta'_1 = {\mathbb R}$ would be the set
of all effect sizes $\delta$, and the group $G$ would be scale
transformation; but there are many other possibilities. To give but a
few examples, \cite{dass-2003-unified-condit} consider testing the
Weibull vs. the log-normal model, the exponential vs. the log-normal,
correlations in multivariate Gaussians, and
\cite{berger-1998-bayes-factors} consider location-scale families and
linear models where $H_0$ and $H_1$ differ in their error
distribution. Importantly, the group $G$ acting on the data induces
groups $G_k$, $k = 0,1$, acting on the parameter spaces, which depend
on the parameterization. In our example, the $G_k$ were equal to $G$,
but, for example, if $H_0$ is Weibull and $H_1$ is log-normal, both
given in their standard parameterizations, we get
$G_0 = \{ g_{0,b,c} \mid g_{0,b,c}(\beta, \gamma) = (b \beta^c, \gamma
/ c), b > 0, c > 0 \}$ and
$G_1 = \{ g_{1, b,c} \mid g_{1, b,c}(\mu, \sigma) = (c\mu+ \log (b), c
\sigma), b > 0, c > 0 \}$. Several more examples are given by
\cite{dass-1998-unified-bayes}.

On the other hand, clearly not all hypothesis sets can be generated
using the above approach. For instance, the hypothesis $H'_1 =
\{P_{\mu, \sigma} \mid \mu=1, \sigma > 0 \}$ with $P_{\mu, \sigma}$ a
Gaussian measure with mean $\mu$ and standard deviation $\sigma$
cannot be represented as in \eqref{eq:standardform}. This is due to
the fact that for $\sigma, \sigma' >0, \sigma\neq \sigma'$, no element
$g \in \R_{>0}$ exists such that for any measurable set $A \subseteq
\samplespacen$ the equality
\begin{align*}
P_{1, \sigma}(A) = P_{1, \sigma'}(A \cdot g^{-1})
\end{align*}
holds. This prevents an equivalent construction of $H'_1$ in the form
of \eqref{eq:standardform}.

We now turn to the main ingredient that will be needed to obtain results on optional stopping: the quotient $\sigma$-algebra.
\begin{definition}[\citet{eaton-1989-group-invar}, Chapter~2]\label{def:orbits, transitive group actions}
	A group $G$ acting on the right of a set $Y$ induces an
	equivalence relation: $y_1 \sim y_2$ if and only if there
	exists $g \in G$ such that $y_1 = y_2 \cdot g$. This equivalence
	relation partitions the space in \emph{orbits}:
	$O_y = \{ y \cdot g \mid g \in G \}$, the collection of which is
	called the \emph{quotient space} $Y / G$.
	There exists a map, the \emph{natural projection}, from $Y$ to the
	quotient space which is defined by
	$\varphi_{Y}: Y \to Y/G: y \mapsto \{y \cdot g \mid g \in G\}$, and which we use to define the {\em quotient $\sigma$-algebra}
	\begin{align}
	\label{eq:quotient-sigma-algebra}
	\G_n =\{\varphi_{\samplespacen}^{-1}(\varphi_{\samplespacen}(A)) \mid A \in \F_n\}.
	\end{align}
\end{definition}
\begin{definition}[\citet{eaton-1989-group-invar}, Chapter~2]
	A random element $U_n$ on $\samplespacen$ is \emph{invariant} if for all
	$g \in G$, $x^n \in \samplespacen$, $U_n(x^n) = U_n(x^n \cdot g)$. The
	random element $U_n$ is \emph{maximal invariant} if $U_n$ is invariant and
	for all $y^n \in \samplespacen$, $U_n(x^n) = U_n(y^n)$ implies
	$x^n = y^n \cdot g$ for some $g \in G$.
	\label{def:invariant}
\end{definition}
Thus, $U_n$ is maximal invariant if and only if $U_n$ is constant on
each orbit, and takes different values on different orbits; $\varphi_{\samplespacen}$
is thus an example of a maximal invariant. Note that any maximal
invariant is $\G_n$-measurable. The importance of this quotient
$\sigma$-algebra $\G_n$ is the following evident fact:
\begin{proposition}
	For fixed $k\in \{0,1 \}$, every invariant $U_n$ has the same
	distribution under all ${P_{k,g}, g \in G}$.
\end{proposition}
Chapter~2 of \citep{eaton-1989-group-invar} provides several methods
and examples how to construct a concrete maximal invariant, including
the first two given below. Since $\beta_n$ is invariant under the
group action of $G$ (see below), $\beta_n$ is an example of an
invariant, although not necessarily of a maximal invariant.

\paragraph{Example~\ref{example:A.1} (continued)}\label{example:A.2}
Consider the setting of the one-sample $t$-test as described above in
Example~\ref{example:A.1}. A maximal invariant for
$x^n \in \samplespacen$ is
\[U_n(x^n) = (x_1 / |x_1|, x_2 / |x_1|, \ldots, x_n / |x_1|).\]

\begin{example}\label{example:B.1}\normalfont
	A second example, with a group invariance structure on two
	parameters, is the setting of the two-sample $t$-test with the
	right Haar prior (which coincides here with Jeffreys' prior)
	\mbox{$\pi(\mu, \sigma) = 1 / \sigma$} (see
	\citet{rouder-2009-bayes} for details): the group is
	\mbox{$G = \{ (a, b) \mid a >0, b \in \mathbb{R} \}$}. Let the sample
	space be \mbox{$\samplespacen = \mathbb{R}^n \setminus \text{span}(e_n)$},
	where $e_n$ denotes a vector of ones of length $n$ (this is to
	exclude the measure-zero line for which the $s(x^n)$ is zero), and
	define the group action by $x^n \cdot (a, b) = ax^n + be_n$ for
	$x^n \in \samplespacen$. Then (\citet{eaton-1989-group-invar}, Example~2.15) a maximal invariant for $x^n \in \samplespacen$ is
	$U_n(x^n) = (x^n - \overline{x}e_n) / s(x^n)$, where $\overline{x}$
	is the sample mean and
	$s(x^n) = \left( \sum_{i=1}^n (x_i - \overline{x})^2 \right)^{1/2}$.
	
	However, we can also construct a maximal invariant similar to the  one in Example~\ref{example:A.1}, which gives  a special status to an initial sample:
	\begin{align*}
	U_n\left(X^n\right) = \left(
	\frac{X_2 - X_1}{|X_2 - X_1|},
	\frac{X_3 - X_1}{|X_2 - X_1|},
	\ldots,
	\frac{X_n - X_1}{ |X_2 - X_1|}\right), \,\, \,\, \, n \geq 2.
	\end{align*}
\end{example}
\subsection{Relatively Invariant Measures and Calibration for Fixed $n$}\label{sec:strongfixed}
Let $U_n$ be a maximal invariant, taking values in the measurable
space $({\cal U}_n, \G_n)$. Although we have given more
concrete examples above, it follows from the results of
\cite{andersson-1982-distr-maxim} that, in case we do not know how to
construct a $U_n$, we can always take $U_n = \varphi_{\samplespacen}$,
the natural projection.  Since we
assume mutual absolute continuity, the Radon-Nikodym derivative
$\frac{\dif \mug{1,g}{U}{n}} {\dif \mug{0,g}{U}{n}}$ must exist and we
can apply the following theorem (note it is here that the use of {\em
	right\/} Haar measure is crucial; a different result holds for the left Haar measure):\footnote{This theorem requires that
	there exists some relatively invariant measure $\mu$ on $\samplespacen$
	such that for $k = 0,1, g \in G$, the $P_{k,g}$ all have a density
	relative to $\mu$. Since the Bayes marginal $\PM$ based on the right
	Haar prior is easily seen to be such a relatively invariant
	measure, the conditions for the theorem apply.}
\paragraph{Theorem \cite[Theorem 2.1]{berger1998bayes}}
Under our previous definitions of and
assumptions on $G$, $P_{k,g}$, $\bar{P}_{k}$ let $\beta(x^n) \coloneqq \bar{P}_1(x^n)/\bar{P}_0(x^n)$ be the Bayes factor based on $x^n$.
Let $U_n$ be a maximal invariant as above, with (adopting the notation of
\eqref{eq:induced}) marginal measures $\mug{k,g}{U}{n}$, for $k =0,1$
and $ g \in G$. There exists a version of the Radon-Nikodym derivative such that we  have for all $g \in G$, all $x^n \in \samplespacen$,
\begin{align}
\label{eq:quotient-likelihood-ration}
\frac{\dif \mug{1,g}{U}{n}}
{\dif \mug{0,g}{U}{n}} \left( U_n(x^n) \right)
= \beta(x^n).
\end{align}
As a first consequence of the theorem above, we note (as did \citet{berger1998bayes}) that the Bayes
factor $\beta_n \coloneqq \beta(X^N)$ is $\G_n$-measurable (it is constant on orbits) , and
thus {\em it has the same distribution under $P_{0,g}$ and $P_{1,g}$
	for all $g \in G$}. The theorem also implies the following crucial lemma:
\begin{lemma}{\bf [Strong Calibration for Fixed $n$]}\label{lem:strongfixedn}
	Under the assumptions of the theorem above, let $U_n$ be a maximal
	invariant and let $V_n$ be a $\G_n$-measurable binary random
	variable with ${P_{0,g}(V_n = 1) > 0}$, ${P_{1,g}(V_n = 1) > 0}$.
	Adopting the notation of \eqref{eq:induced}, we can choose the
	Radon-Nikodym derivative ${\dif \mugcond{1,g}{\beta}{n}{V_n = 1}}/
	{\dif \mugcond{0,g}{\beta}{n}{V_n= 1}}$ so that we have, for all
	${x^n \in \samplespacen}$:
	\begin{align}\label{eq:thesame}
	\frac{P_{1,g}(V_n = 1)}{P_{0,g}(V_n = 1)} \cdot
	\frac{\dif \mugcond{1,g}{\beta}{n}{V_n = 1}}
	{\dif \mugcond{0,g}{\beta}{n}{V_n= 1}}(\beta_n(x^n)) = \beta_n(x^n),
	\end{align}
	where for the special case with $P_{k,g}(V_n  = 1) = 1$, we get
	$\frac{\dif \mug{1,g}{\beta}{n}}
	{\dif \mug{0,g}{\beta}{n}}(\beta_n(x^n)) = \beta_n(x^n)$.
\end{lemma}
\subsection{Extending to Our General Setting with Non-Fixed Sample Sizes}
\label{sec:extending}
We start with the same setting as above: a group $G$ on sample space
$\samplespacen \subset \cX^n$ that acts topologically and properly on
the right of $\samplespacen$; two distributions $P_{0,e}$ and
$P_{1,e}$ on $(\samplespacen,\F_n)$ that are used to generate $H_0$
and $H_1$, and Bayes marginal measures based on the right Haar measure
$\PM$ and $\QM$, which are both $\sigma$-finite. We now denote $H_k$
as $H_k^{(n)}$, $P_{k,e}$ as $P_{k,e}^{(n)}$ and $\KM$ as $\KM^{(n)}$,
all $P \in H_0^{(n)} \cup H_1^{(n)}$ are mutually absolutely
continuous.

We now extend this setting to our general random process setting as specified in the beginning of
Section~\ref{sec:generalcalibration} by further assuming that, for the same group
$G$, for some $m > 0$, the above setting is defined for each $n \geq
m$. To connect the $H_k^{(n)}$ for all these $n$, we further assume that
there exists a subset $\initialspace \subset \cX^m$ that
has measure $1$ under $P^{(n)}_{k,e}$  (and hence under all
$P^{(n)}_{g,e}$) such that for all $n \geq m$:
\begin{enumerate}
	\item We can write $\samplespacen = \{x^n \in
	\cX^n: (x_1, \ldots, x_m) \in \langle \cX^m \rangle \}$.
	\item For all $x^n \in \samplespacen$, the  posterior $\nu \mid x^n$ based on the right Haar measure $\nu$ is proper.
	\item The probability measures $P^{(n)}_{k,e}$ and $P^{(n+1)}_{k,e}$ satisfy Kolmogorov's compatibility condition for a random process.
	\item The group action $\cdot$ on the measures $P^{(n)}_{k,e}$ and
	$P^{(n+1)}_{k,e}$ is compatible, i.e.\ for every $n >0$, for every $A
	\in \F_n$, every $g \in G$, $k \in \{0,1\}$, we have
	$P^{(n+1)}_{k,g}(A) = P^{(n)}_{k,g}(A)$.
\end{enumerate}
Requirement 4.\ simply imposes the condition that the group action considered is
the same for all $n \in {\mathbb N}$. As a consequence of 3.\ and 4.,
the probability measures $P^{(n)}_{k,g}$ and $P^{(n+1)}_{k,g}$ satisfy
Kolmogorov's compatibility condition for all $g \in G$,
$k \in \{0,1\}$ which means that there exists a probability measure
$P_{k,g}$ on $(\Omega, \F)$ (under which
$\initialX, X_{m+1}, X_{m+2}, \ldots$ is a random process), defined as
in the beginning of Section~\ref{sec:general}, whose marginals for
$n \geq m$ coincide with $P^{(n)}_{k,g}$, and there exist measures
$\PM$ and $\QM$ on $(\Omega, \F)$ whose marginals for $n \geq m$
coincide with $\PM^{(n)}$ and $\QM^{(n)}$.  We have thus defined a set
$H_0$ and $H_1$ of hypotheses on $(\Omega, \F)$ and the corresponding
Bayes marginals $\PM$ and $\QM$ and are back in our general
setting. It is easily verified that the 1- and 2-sample Bayesian
$t$-tests both satisfy all these assumptions: in
Example~\ref{example:A.1}, take $m=1$ and
$\initialspace = {\mathbb R} \setminus \{0\}$; in
Example~\ref{example:A.2}, take $m=2$ and
$\initialspace = {\mathbb R}^2 \setminus \{(a,a): a \in {\mathbb R
}\}$. The conditions can also be verified for the variety of examples
considered by \citet{berger1998bayes} and \cite{bayarri2012criteria}.
In fact, our initial sample $x^m \in \initialspace$ is a variation of
what they call a {\em minimal sample\/}; by excluding `singular'
outcomes from $\cX^m$ to ensure that the group acts properly on
$\initialspace$, we can guarantee that the initial sample is of fixed
size. The size of the minimal sample can be larger, on a set of
measure 0 under all $P \in H_0 \cup H_1$, e.g.\ if, in
Example~\ref{example:A.2}, $X_1 = X_2$.  We chose to ensure a fixed
size $m$ since it makes the extension to random processes considerably
easier.

In Section \ref{sec:general-background}, underneath
Example~\ref{example:A.1} we already outlined how a composite
alternative hypothesis can be reduced to a hypothesis with just a free
nuisance parameter (or parameter vector) $g\in G$, by putting a proper prior on all other
parameters and integrating them out. A similar construction for a
single parameter alternative hypothesis in the form of
\eqref{eq:standardform} can be applied in the non-fixed sample size
case.

\subsection{Strong Calibration}
\label{sec:strongcalibration}
Consider the setting, definitions and assumptions of the previous
subsection, with the additional assumptions and definitions made in
the beginning of Section~\ref{sec:genfreq}, in particular the
assumption of a.s.\ finite stopping time. For simplicity, from now on, we shall also assume equal prior odds, $\pi(H_0) = \pi(H_1) = 1/2$. We will
now show a strong calibration theorem for the Bayes factors $\beta_n =
(d \PM^{(n)}) / (d \QM^{(n)})(X^n)$ defined in terms of the Bayes
marginals $\PM$ and $\QM$ with the right Haar prior.  Thus
$\beta_{\tau}$ is defined as in \eqref{eq:def-Bt} with $\beta$ in the
role of $B$.
\begin{theorem}[Strong calibration under optional stopping]
	\label{thm:strong-calibration-under-os}
	Let $\tau$ be a stopping time satisfying our requirements, such
	that additionally, for each $n > m$, the event $\{ \tau = n \}$ is
	$\G_n$-measurable. Then, adopting the notation of
	\eqref{eq:induced}, for all $g \in G$, for
	$\mug{0,g}{\beta}{\tau}$-almost every $b > 0$, we have:
	\begin{align}\nonumber
	\frac{\dif \mug{1,g}{\beta}{\tau}}{\dif \mug{0,g}{\beta}{\tau}}(b) = b.
	\end{align}
	That means that the posterior odds remain calibrated under every stopping
	rule $\tau$ adapted to the quotient space filtration $\G_m, \G_{m+1}, \ldots$, under all $P_{0,g}$.
\end{theorem}

\begin{proof}
	%
	Fix some $g \in G$. We simply first apply Lemma~\ref{lem:strongfixedn}
	with $V_n = \ind[\tau =n]$, which gives that the premise
	\eqref{eq:ratio-OS-minibatch} of Lemma~\ref{lem:caligali} holds with
	$c=1$ and $\beta_n$ in the role of $B_n$ (it is here that we need that $\tau_n$ is $\G_n$-measurable, otherwise we could not apply
	Lemma~\ref{lem:strongfixedn} with the required definition of $V_n$).
	We can now use
	Lemma~\ref{lem:caligali}
	with $P_{0,g}$ in the role of $P_0$ to reach the desired conclusion for the chosen $g$. Since this works for all $g \in G$, the result follows.
\end{proof}

\paragraph{Example~\ref{example:A.1}, Continued: Admissible and
	Inadmissible Stopping Rules} We obtain strong calibration for the
one-sample $t$-test with respect to the nuisance parameter $\sigma$ (see
Example~\ref{example:A.1} above) when the stopping rule is adapted to
the quotient filtration $\G_m, \G_{m+1}, \ldots$. Under each
$P_{k,g} \in H_k$, the Bayes factors $\beta_m, \beta_{m+1}, \ldots$
define a random process on $\Omega $ such that each $\beta_n$ is
$\G_n$-measurable.  This means that a stopping time defined in terms
of a rule such as `stop at the smallest $t$ at which $\beta_t > 20$ or
$t = 10^6$' is allowed in the result above.  Moreover, if the stopping rule is a
function of a sequence of maximal invariants, like
$x_1 / |x_1|, x_2 / |x_1|, \ldots$, it is adapted to the filtration
$\G_m, \G_{m+1}, \ldots$
and we can likewise apply the result above. On the other hand, this
requirement is violated, for example, by a stopping rule that stops
when $\sum_{i=1}^j (x_i)^2$ exceeds some fixed value, since such a
stopping rule explicitly depends on the scale of the sampled data.

\subsection{Frequentist optional stopping}
\label{sec:strongfrequentist}
The special case of the following result for the one-sample Bayesian $t$-test
was proven in the master's thesis  \citep{hendriksen-2017-betting-as}. Here we extend the result to general group invariances.
\begin{theorem}[Frequentist optional stopping for composite null
	hypotheses with group invariance]\label{th:frequentist-os-with-nuisance-pars}
	Under the same conditions as in Section~\ref{sec:strongcalibration},  let $\tau$ be a stopping time such
	that, for each $n > m$, the event $\{ \tau = n \}$ is
	$\G_n$-measurable. Then, adopting the notation of
	\eqref{eq:induced}, for all $g \in G$, the stopped Bayes factor satisfies
	${\bf E}_{P_{0,g}}[ \beta_{\tau}] = \int_{{\mathbb R}_{> 0}}c   \dif \mug{0,g}{\beta}{\tau}(c) = 1$, so that, by the reasoning above Proposition~\ref{prop:safetest}, we have {\em for all $g \in G$:} $P_{0,g}(\frac1{\beta_\tau} \leq \alpha)  \leq \alpha$.
\end{theorem}
\begin{proof}
	We have
	\begin{align*}
	\int_{{\mathbb R}_{> 0}} \ c \  d \mug{0,g}{\beta}{\tau}(c)
	&= \int_{{\mathbb R}_{> 0}} \frac{\dif \mug{1,g}{\beta}{\tau}}{\dif \mug{0,g}{\beta}{\tau} } (c)   d \mug{0,g}{\beta}{\tau}(c)
	= \int_{{\mathbb R}_{> 0}} d\mug{1,g}{\beta}{\tau}(c) = 1.
	\end{align*}
	where the first equality follows directly from Theorem~\ref{thm:strong-calibration-under-os} and the final equality follows because $P_{1,g}$ is a probability measure, integrating to 1.
\end{proof}
Analogously to Corollary~\ref{cor:thailand}, the desired result now follows by plugging in a particular
stopping rule: let \(S: \bigcup_{i=m}^{\infty}  \X^i \to \{0,1\}\) be
the frequentist sequential test  defined by setting, for all $n> m$,
$x^n \in \samplespacen$:  $S(x^n) = 1$ if and only if $\beta_n \geq 1/\alpha$.
\begin{corollary}\label{cor:bali}
	Let $t^* \in \{m+1, m+2, \ldots \} \cup \{ \infty \}$ be the smallest $t^* > m$ for which $\beta_{t^*}^{-1} \leq \alpha$. Then for  arbitrarily large $T$,
	when applied to the stopping rule $\tau \coloneqq \min \{T, t^*\}$, we find that {\em for all $g \in G$}:
	$$P_{0,g}(\exists n, m < n \leq T: S(X^n) = 1 \mid x^m)
	=P_{0,g}(\exists n, m < n \leq T: \beta^{-1}_n \leq \alpha \mid x^m)
	\leq \alpha.$$
\end{corollary}
The corollary implies that the test $S$ is robust under optional stopping in the frequentist sense relative to $H_0$  (Definition~\ref{def:freqrobust}).

\paragraph{Example~\ref{example:A.1} (continued)}
When we choose a stopping rule that is $(\G_m, \G_{m+1}, \ldots)$-measurable,
the hypothesis test is robust under (semi-)frequentist
optional stopping. This holds for example, for the one- and two-sample $t$-test
\citep{rouder-2009-bayes}, Bayesian ANOVA \citep{rouder-2012-default-bayes},
and Bayesian linear regression \citep{liang2008mixtures}. Again, for
stopping rules that are not  $(\G_m, \G_{m+1}, \ldots)$-measurable,
robustness under frequentist optional stopping cannot be guaranteed
and could reasonably be presumed to be violated. The violation of
robustness under optional stopping is hard to demonstrate
experimentally as frequentist Bayes factor tests are usually quite
conservative in approaching the asymptotic significance level
$\alpha$.

\section{Concluding Remarks}\label{sec:concluding remarks}
We have identified three types of `handling optional stopping':
$\tau$-independence, calibration and semi-frequentist. We extended the
corresponding definitions and results to general sample spaces with
potentially improper priors. For the special case of models $H_0$ and
$H_1$ sharing a nuisance parameter with a group invariance structure,
we showed stronger versions of calibration and semi-frequentist
robustness to optional stopping. A couple of remarks are in
order. First, one of the remarkable properties of the right Haar prior
is that, under some additional conditions on $P_{0,g}$ and $P_{1,g}$
in \eqref{eq:definitiePkg}, $\beta_m = \beta(x^m)=1$ for all
$x^m \in \initialspace$, implying that equal prior odds lead to equal
posterior odds after a minimal sample, no matter what the minimal
sample is \citep{berger-1998-bayes-factors}. One might conjecture that
our results rely on this property, but this is not the case: in
general, one can have $\beta(x^m)\neq 1$, yet our results still hold.
For example, in the Bayesian $t$-test, Example~\ref{example:A.1},
$m=1$ and $\beta(x^1)=1$ can be guaranteed only if the prior
$\pi_{\delta}$ on $\delta$ is symmetric around $0$; but our
calibration and frequentist robustness results hold irrespective of
whether it is symmetric or not.

Secondly, in multiple-parameter problems, the suitable transformation group acting on the parameter space may not be unique, in which case there are multiple possible right Haar priors, see Example~1.2 and 1.3 in \citep{berger2015overall} and \citep{berger2008objective}. However, in all examples we considered and further know of, this does not lead to ambiguity, because different transformation groups give rise to different sets $H_0$ of invariant null hypotheses.

As a third remark, it is worth
noting that --- as is immediate from the proofs --- all our
group-invariance results continue to hold in the setting with $H'_k$ as
in \eqref{eq:there}, and the definition of the Bayes marginal
$P_{k,e}$ relative to $\theta'$ as in \eqref{eq:almost} replaced by a
probability measure on $(\Omega,\F)$ that is not necessarily of the
Bayes marginal form. The results work for any probability measure; in
particular one can take the alternatives for the Bayes marginal with
proper prior that are considered in the the minimum description length
and sequential prediction literature \citep{BarronRY98,Grunwald07}
under the name of {\em universal distribution\/} relative to
$\{P_{\theta'}\mid \theta' \in \Theta' \}$; examples include the
prequential or `switch' distributions
considered by \cite{PasG18}.

As a fourth and final remark, a sizable fraction of Bayesian
statisticians is wary of using improper priors at all. An important
(though not the only) reason is that their use often leads to some
form of the {\em marginalization paradox\/} described by
\cite{dawid1973marginalization}. It is thus useful to stress that in
the context of Bayes factor hypothesis testing, the right Haar prior
is immune at least to this particular paradox. In an informal
nutshell, the marginalization paradox occurs if the following happens:
(a) the Bayes posterior $\pi(\zeta \mid X^n)$ for the quantity of
interest $\zeta$ based on prior $\pi(\zeta,g)$ with improper marginal
on $g$, only depends on the data $X^{n}$ through the maximal invariant
$U_n$, i.e. $\pi(\zeta \mid X^n) = f(U_n(X^n))$ for some function $f$,
yet (b) there exists no prior $\pi'$ on $\zeta$ such that the
corresponding posterior $\pi'(\zeta \mid U_n(X^n)) = f(U_n(X^n))$.  In
words, the result of Bayesian updating based on the full data $X^n$
only depends on the maximal invariant $U^n$; but Bayesian updating
directly based on $U^n$ can never give the same result --- a paradox
indeed. While in general, this can happen even if $g$ is equipped with
the right Haar prior [Case 1, page
199]\citep{dawid1973marginalization}, Berger et.\ al.'s Theorem 2.1
(reproduced in Section~\ref{sec:strongfixed} in our paper) implies
that it does not occur in the context of Bayes factor testing, where
$\zeta \in \{H_0, H_1\}$, and $H_0$ and $H_1$ are null and
alternatives satisfying the requirements of
Section~\ref{sec:group}. Berger's theorem expresses that for all
values of the nuisance parameter $g \in G$, the likelihood ratio
${\dif \mug{1,g}{U}{n}} /{\dif \mug{0,g}{U}{n}} \left( U_n(X^n)
\right)$ based on $U_n(X^n)$ is equal to the Bayes factor based on
$X^n$ with the right Haar prior on $g$, so that the paradox cannot
occur.

\paragraph{Acknowledgements}   We are grateful to Wouter Koolen, for extremely useful conversations that helped with the math, and to Jeff Rouder, for providing inspiration and insights that sparked off this research, and to Aaditya Ramdas, who brought \cite{lai1976confidence} to our attention.

\DeclareRobustCommand{\VANDER}[3]{#3}
\bibliographystyle{abbrvnat}
\bibliography{bib/martingale-bib}

\begin{thebibliography}{42}
\providecommand{\natexlab}[1]{#1}
\providecommand{\url}[1]{\texttt{#1}}
\expandafter\ifx\csname urlstyle\endcsname\relax
  \providecommand{\doi}[1]{doi: #1}\else
  \providecommand{\doi}{doi: \begingroup \urlstyle{rm}\Url}\fi

\bibitem[Andersson(1982)]{andersson-1982-distr-maxim}
S.~Andersson.
\newblock Distributions of maximal invariants using quotient measures.
\newblock \emph{The Annals of Statistics}, 10\penalty0 (3):\penalty0 955--961,
  1982.
\newblock ISSN 00905364.

\bibitem[Barnard(1947)]{barnard1947review}
G.~A. Barnard.
\newblock Review of {\em {s}equential {a}nalysis\/} by {A}braham {W}ald.
\newblock \emph{Journal of the American Statistical Association}, 42\penalty0
  (240), 1947.

\bibitem[Barnard(1949)]{barnard1949statistical}
G.~A. Barnard.
\newblock Statistical inference.
\newblock \emph{Journal of the Royal Statistical Society. Series B
  (Methodological)}, 11\penalty0 (2):\penalty0 115--149, 1949.

\bibitem[Barron et~al.(1998)Barron, Rissanen, and Yu]{BarronRY98}
A.~Barron, J.~Rissanen, and B.~Yu.
\newblock The minimum description length principle in coding and modeling.
\newblock \emph{IEEE Transactions on Information Theory}, 44\penalty0
  (6):\penalty0 2743--2760, 1998.
\newblock \doi{10.1109/18.720554}.

\bibitem[Bayarri et~al.(2012)Bayarri, Berger, Forte, and
  Garc{\'\i}a-Donato]{bayarri2012criteria}
M.~J. Bayarri, J.~O. Berger, A.~Forte, and G.~Garc{\'\i}a-Donato.
\newblock Criteria for {B}ayesian model choice with application to variable
  selection.
\newblock \emph{The Annals of statistics}, 40\penalty0 (3):\penalty0
  1550--1577, 2012.

\bibitem[Bayarri et~al.(2016)Bayarri, Benjamin, Berger, and
  Sellke]{bayarri2016rejection}
M.~J. Bayarri, D.~J. Benjamin, J.~O. Berger, and T.~M. Sellke.
\newblock Rejection odds and rejection ratios: A proposal for statistical
  practice in testing hypotheses.
\newblock \emph{Journal of Mathematical Psychology}, 72:\penalty0 90--103,
  2016.

\bibitem[Berger and Wolpert(1988)]{BergerW88}
J.~O. Berger and R.~L. Wolpert.
\newblock \emph{The Likelihood Principle}.
\newblock Institute of Mathematical Statistics, Hayward, CA, 2nd edition, 1988.

\bibitem[Berger et~al.(1998{\natexlab{a}})Berger, Pericchi, and
  Varshavsky]{berger-1998-bayes-factors}
J.~O. Berger, L.~R. Pericchi, and J.~A. Varshavsky.
\newblock Bayes factors and marginal distributions in invariant situations.
\newblock \emph{Sankhy{\=a}: The Indian Journal of Statistics, Series A}, pages
  307--321, 1998{\natexlab{a}}.

\bibitem[Berger et~al.(1998{\natexlab{b}})Berger, Pericchi, and
  Varshavsky]{berger1998bayes}
J.~O. Berger, L.~R. Pericchi, and J.~A. Varshavsky.
\newblock Bayes factors and marginal distributions in invariant situations.
\newblock \emph{Sankhy{\=a}: The Indian Journal of Statistics, Series A}, pages
  307--321, 1998{\natexlab{b}}.

\bibitem[Berger et~al.(2008)Berger, Sun, et~al.]{berger2008objective}
J.~O. Berger, D.~Sun, et~al.
\newblock Objective priors for the bivariate normal model.
\newblock \emph{The Annals of Statistics}, 36\penalty0 (2):\penalty0 963--982,
  2008.

\bibitem[Berger et~al.(2015)Berger, Bernardo, Sun, et~al.]{berger2015overall}
J.~O. Berger, J.~M. Bernardo, D.~Sun, et~al.
\newblock Overall objective priors.
\newblock \emph{Bayesian Analysis}, 10\penalty0 (1):\penalty0 189--221, 2015.

\bibitem[Conway(2013)]{conway-2013-course-in}
J.~B. Conway.
\newblock \emph{A course in functional analysis}, volume~96.
\newblock Springer Science \& Business Media, 2013.

\bibitem[Dass(1998)]{dass-1998-unified-bayes}
S.~C. Dass.
\newblock \emph{Unified {B}ayesian and conditional frequentist testing
  procedures}.
\newblock PhD thesis, University of Michigan, 1998.

\bibitem[Dass and Berger(2003)]{dass-2003-unified-condit}
S.~C. Dass and J.~O. Berger.
\newblock Unified conditional frequentist and {B}ayesian testing of composite
  hypotheses.
\newblock \emph{Scandinavian Journal of Statistics}, 30\penalty0 (1):\penalty0
  193–210, Mar 2003.
\newblock ISSN 1467-9469.
\newblock \doi{10.1111/1467-9469.00326}.

\bibitem[Dawid et~al.(1973)Dawid, Stone, and Zidek]{dawid1973marginalization}
A.~P. Dawid, M.~Stone, and J.~V. Zidek.
\newblock Marginalization paradoxes in bayesian and structural inference.
\newblock \emph{Journal of the Royal Statistical Society: Series B
  (Methodological)}, 35\penalty0 (2):\penalty0 189--213, 1973.

\bibitem[Deng et~al.(2016)Deng, Lu, and Chen]{deng2016continuous}
A.~Deng, J.~Lu, and S.~Chen.
\newblock Continuous monitoring of {A/B} tests without pain: Optional stopping
  in {B}ayesian testing.
\newblock In \emph{Data Science and Advanced Analytics (DSAA), 2016 IEEE
  International Conference on}, pages 243--252. IEEE, 2016.

\bibitem[Eaton(1989)]{eaton-1989-group-invar}
M.~L. Eaton.
\newblock Group invariance applications in statistics.
\newblock \emph{Regional Conference Series in Probability and Statistics},
  1:\penalty0 i--133, 1989.
\newblock ISSN 19355912.

\bibitem[Edwards et~al.(1963)Edwards, Lindman, and
  Savage]{edwards-1963-bayes-statis}
W.~Edwards, H.~Lindman, and L.~J. Savage.
\newblock Bayesian statistical inference for psychological research.
\newblock \emph{Psychological Review}, 70\penalty0 (3):\penalty0 193–242,
  1963.
\newblock ISSN 0033-295X.
\newblock \doi{10.1037/h0044139}.

\bibitem[Good(1991)]{Good1991}
I.~J. Good.
\newblock C383. {A} comment concerning optional stopping.
\newblock \emph{Journal of Statistical Computation and Simulation}, 39\penalty0
  (3):\penalty0 191--192, 1991.

\bibitem[Gr{\"u}nwald et~al.(2019)Gr{\"u}nwald, {\VANDER{Heide}{De}{de}}~Heide,
  and Koolen]{safetesting}
P.~Gr{\"u}nwald, R.~{\VANDER{Heide}{De}{de}}~Heide, and W.~Koolen.
\newblock Safe testing.
\newblock \emph{arXiv preprint arXiv:1906.07801}, 2019.

\bibitem[Gr\"unwald(2007)]{Grunwald07}
P.~D. Gr\"unwald.
\newblock \emph{The Minimum Description Length Principle}.
\newblock MIT Press, Cambridge, MA, 2007.

\bibitem[{\VANDER{Heide}{De}{de}}~Heide and
  Gr{\"u}nwald(2018)]{heide-2017-why-option}
R.~{\VANDER{Heide}{De}{de}}~Heide and P.~Gr{\"u}nwald.
\newblock Why optional stopping is a problem for {B}ayesians.
\newblock \emph{arXiv preprint arXiv:1708.08278}, 2018.

\bibitem[Hendriksen(2017)]{hendriksen-2017-betting-as}
A.~A. Hendriksen.
\newblock Betting as an alternative to $p$-values.
\newblock Master's thesis, Leiden University, Dept. of Mathematics, 2017.

\bibitem[Howard et~al.(2018)Howard, Ramdas, McAuliffe, and
  Sekhon]{howard2018uniform}
S.~R. Howard, A.~Ramdas, J.~McAuliffe, and J.~Sekhon.
\newblock Uniform, nonparametric, non-asymptotic confidence sequences.
\newblock \emph{arXiv preprint arXiv:1810.08240}, 2018.

\bibitem[Jamil et~al.(2016)Jamil, Ly, Morey, Love, Marsman, and
  Wagenmakers]{jamil-2016-default-gunel}
T.~Jamil, A.~Ly, R.~D. Morey, J.~Love, M.~Marsman, and E.-J. Wagenmakers.
\newblock Default "gunel and dickey" bayes factors for contingency tables.
\newblock \emph{Behavior Research Methods}, 49\penalty0 (2):\penalty0 638--652,
  Jun 2016.
\newblock ISSN 1554-3528.
\newblock \doi{10.3758/s13428-016-0739-8}.

\bibitem[Jeffreys(1961)]{jeffreys-1961-theory-of}
H.~Jeffreys.
\newblock \emph{Theory of Probability}.
\newblock Oxford, Oxford, England, 1961.

\bibitem[John et~al.(2012)John, Loewenstein, and
  Prelec]{john-2012-measur-preval}
L.~K. John, G.~Loewenstein, and D.~Prelec.
\newblock Measuring the prevalence of questionable research practices with
  incentives for truth telling.
\newblock \emph{Psychological science}, 2012.

\bibitem[Lai(1976)]{lai1976confidence}
T.~L. Lai.
\newblock On confidence sequences.
\newblock \emph{The Annals of Statistics}, 4\penalty0 (2):\penalty0 265--280,
  1976.

\bibitem[Liang et~al.(2008)Liang, Paulo, Molina, Clyde, and
  Berger]{liang2008mixtures}
F.~Liang, R.~Paulo, G.~Molina, M.~A. Clyde, and J.~O. Berger.
\newblock Mixtures of g priors for {B}ayesian variable selection.
\newblock \emph{Journal of the American Statistical Association}, 103\penalty0
  (481):\penalty0 410--423, 2008.
\newblock \doi{10.1198/016214507000001337}.

\bibitem[Lindley(1957)]{lindley-1957-statis-parad}
D.~V. Lindley.
\newblock A statistical paradox.
\newblock \emph{Biometrika}, 44\penalty0 (1/2):\penalty0 187--192, Jun 1957.
\newblock ISSN 0006-3444.
\newblock \doi{10.2307/2333251}.

\bibitem[{\VANDER{Pas}{Van der}{van der}}~Pas and Gr\"unwald(2018)]{PasG18}
S.~{\VANDER{Pas}{Van der}{van der}}~Pas and P.~D. Gr\"unwald.
\newblock Almost the best of three worlds: Risk, consistency and optional
  stopping for the switch criterion in nested model selection.
\newblock \emph{Statistica Sinica}, 28\penalty0 (1):\penalty0 229--253, 2018.

\bibitem[Proschan et~al.(2006)Proschan, Lan, and
  Wittes]{proschan2006statistical}
M.~A. Proschan, K.~G. Lan, and J.~T. Wittes.
\newblock \emph{Statistical monitoring of clinical trials: a unified approach}.
\newblock Springer Science \& Business Media, 2006.

\bibitem[Raiffa and Schlaifer(1961)]{RaiffaS61}
H.~Raiffa and R.~Schlaifer.
\newblock \emph{Applied Statistical Decision Theory}.
\newblock Harvard University Press, Cambridge, MA, 1961.

\bibitem[Rouder(2014)]{rouder-2014-option}
J.~N. Rouder.
\newblock Optional stopping: No problem for {B}ayesians.
\newblock \emph{Psychonomic Bulletin \& Review}, 21\penalty0 (2):\penalty0
  301–308, Mar 2014.
\newblock ISSN 1531-5320.
\newblock \doi{10.3758/s13423-014-0595-4}.

\bibitem[Rouder et~al.(2009)Rouder, Speckman, Sun, Morey, and
  Iverson]{rouder-2009-bayes}
J.~N. Rouder, P.~L. Speckman, D.~Sun, R.~D. Morey, and G.~Iverson.
\newblock Bayesian t tests for accepting and rejecting the null hypothesis.
\newblock \emph{Psychonomic Bulletin \& Review}, 16\penalty0 (2):\penalty0
  225–237, Apr 2009.
\newblock ISSN 1531-5320.
\newblock \doi{10.3758/pbr.16.2.225}.

\bibitem[Rouder et~al.(2012)Rouder, Morey, Speckman, and
  Province]{rouder-2012-default-bayes}
J.~N. Rouder, R.~D. Morey, P.~L. Speckman, and J.~M. Province.
\newblock Default {B}ayes factors for {ANOVA} designs.
\newblock \emph{Journal of Mathematical Psychology}, 56\penalty0 (5):\penalty0
  356--374, 2012.

\bibitem[Sanborn and Hills(2014)]{sanborn-2013-frequen-implic}
A.~N. Sanborn and T.~T. Hills.
\newblock The frequentist implications of optional stopping on {B}ayesian
  hypothesis tests.
\newblock \emph{Psychonomic Bulletin \& Review}, 21\penalty0 (2):\penalty0
  283–300, Oct 2014.
\newblock ISSN 1531-5320.
\newblock \doi{10.3758/s13423-013-0518-9}.

\bibitem[Sch\"{o}nbrodt et~al.(2017)Sch\"{o}nbrodt, Wagenmakers, Zehetleitner,
  and Perugini]{schonbrodt-2017-sequen-hypot}
F.~D. Sch\"{o}nbrodt, E.-J. Wagenmakers, M.~Zehetleitner, and M.~Perugini.
\newblock Sequential hypothesis testing with {B}ayes factors: Efficiently
  testing mean differences.
\newblock \emph{Psychological Methods}, 22\penalty0 (2):\penalty0 322–339,
  2017.
\newblock ISSN 1082-989X.
\newblock \doi{10.1037/met0000061}.

\bibitem[Shafer et~al.(2011)Shafer, Shen, Vereshchagin, and
  Vovk]{shafer2011test}
G.~Shafer, A.~Shen, N.~Vereshchagin, and V.~Vovk.
\newblock Test martingales, {B}ayes factors and p-values.
\newblock \emph{Statistical Science}, 26\penalty0 (1):\penalty0 84--101, 02
  2011.
\newblock \doi{10.1214/10-STS347}.

\bibitem[Wagenmakers(2007)]{wagenmakers-2007-pract-solut}
E.-J. Wagenmakers.
\newblock A practical solution to the pervasive problems of p values.
\newblock \emph{Psychonomic Bulletin \& Review}, 14\penalty0 (5):\penalty0
  779–804, Oct 2007.
\newblock ISSN 1531-5320.
\newblock \doi{10.3758/bf03194105}.

\bibitem[Wijsman(1990)]{wijsman-1990-invar}
R.~A. Wijsman.
\newblock \emph{Invariant measures on groups and their use in statistics}.
\newblock Institute of Mathematical Statistics, 1990.
\newblock ISBN 9780940600195.

\bibitem[Yu et~al.(2014)Yu, Sprenger, Thomas, and
  Dougherty]{yu-2013-when-decis}
E.~C. Yu, A.~M. Sprenger, R.~P. Thomas, and M.~R. Dougherty.
\newblock When decision heuristics and science collide.
\newblock \emph{Psychonomic Bulletin \& Review}, 21\penalty0 (2):\penalty0
  268–282, Sep 2014.
\newblock ISSN 1531-5320.
\newblock \doi{10.3758/s13423-013-0495-z}.

\end{thebibliography}

\appendix

\section{Group theoretic preliminaries}\label{app:groupprel}

We start with some group-theoretical preliminaries; for more details,
see e.g.\  \citep{eaton-1989-group-invar, wijsman-1990-invar,
	andersson-1982-distr-maxim}.

\begin{definition}[Topological space] A non-empty set $S$ together with a fixed collection of subsets $\mathcal{T}$ is called a \emph{topological space} $T = (S, \mathcal{T})$ if
	\begin{enumerate}[(i)]
		\item $S, \emptyset \in \mathcal{T}$,
		\item $U \cap V \in \mathcal{T}$ for any two sets $U, V \in \mathcal{T}$, and
		\item $S_1 \cup S_2 \in \mathcal{T}$ for any collections of sets $S_1, S_2 \subseteq \mathcal{T}$.
	\end{enumerate}
\end{definition}
The collection $\mathcal{T}$ is called a \emph{topology} for $S$, and its members are called the \emph{open sets} of $T$. A topological space $T$ is called \emph{Hausdorff} if for any two distinct points $x, y \in T$ there exist disjoint open subsets $U, V$ of $T$ containing one point each.

\begin{definition}[(Local) compactness] A topological space $T$ is \emph{compact} if every \emph{open cover}, that is, every collection $\C$ of open sets of $T$
	\begin{align*}
	T = \bigcup_{U \in \C} U,
	\end{align*}
	has a \emph{finite subcover}: a finite subcollection $\mathcal{F} \in \C$ such that
	\begin{align*}
	T = \bigcup_{V \in \mathcal{F}} V.
	\end{align*}
	It is \emph{locally compact} if for every $x \in T$ there exist an open set $U$ such that $x \in U$ and the closure of $U$, denoted by $\text{Cl}(U)$, is compact, that is, the union of $U$ and all its limit points in $T$ is compact. We can also formulate this as each $x$ having a neighborhood $U$ such that $\text{Cl}(U)$ is compact.
\end{definition}

\begin{example}[Locally compact Hausdorff spaces] The reals $\mathbb{R}$ and the Euclidean spaces $\mathbb{R}^n$ together with the Euclidean topology (also called the \emph{usual topology}) are \emph{locally compact Hausdorff spaces}. $\mathbb{R}^n$ (for $n \in \mathbb{N}$) is locally compact because any open ball $B(x, r)$ has a compact closure $\text{Cl}(B(x,r)) = \{y \in \mathbb{R}^n ; d(x,y) \leq \epsilon \}$, where $d(x,y)$ is the Euclidean metric. Any discrete space is locally compact and Hausdorff as well, as any singleton is a neighborhood that equals its closure, and it is compact only if it is finite. Infinite dimensional Banach spaces (function spaces) are for example not locally compact.
\end{example}

\begin{definition}[Group] A set $G$ together with a binary operation $\circ$, often called the \emph{group law}, is called a \emph{group} when
	\begin{enumerate}[(i)]
		\item there exists an identity element $e \in G$ for the group law $\circ$,
		\item for every three elements $a, b, c \in G$, we have $(a \circ b) \circ c = a \circ (b \circ c)$ (associativity), and
		\item for each element $a \in G$, there exists an inverse element, $a^\dagger \in G$, with $a \circ a^\dagger = a^\dagger \circ a = e$.
	\end{enumerate}
\end{definition}

\paragraph{Transformation groups} A group that consists of a set $G$ of transformations on some set $S$ is called a \emph{transformation group}. We also say that \emph{the group $G$ acts on the set $S$}. A transformation is a mapping from $S$ to itself that preserves certain properties, such as isometries in the Euclidean plane. Transformation groups are usually not commutative, that is $a \circ b \neq b \circ a$ for $a, b \in G$.

\begin{definition}[Topological group] A topological space $G$ that is also a group is called a \emph{topological group} when the group operation $\circ$ is continuous, that is, for $a, b \in G$, we have that the operations of product
	\begin{enumerate}[(i)]
		\item $G \times G \rightarrow G : (a, b) \mapsto a \circ b$, and taking the inverse
		\item $G \rightarrow G$ : $a \mapsto a^\dagger$,
	\end{enumerate}
	are continuous, where $G \times G$ has the product topology.
\end{definition}
A topological group for which the underlying topology is locally compact and Hausdorff, is called a \emph{locally compact group}.

\begin{definition}[\citet{eaton-1989-group-invar}, Definition 2.1]
	Let $Y$ be a set, and let \(G\) be a group with identity element
	\(e\).  A function $F: Y \times G \to Y$ satisfying
	\begin{enumerate}
		\item $F(y,e) = y, \,\,\, y \in Y$
		\item $F(y, g_1 g_2) = F(F(y, g_1), g_2),\,\,\, g_1, g_2 \in G, y \in Y$
	\end{enumerate}
	specifies $G$ acting on the right of $Y$.
\end{definition}
In practice, $F$ is omitted: we will write $y \cdot g$ for a group
element $g$ acting on the right of $y \in Y$. For a subset
$A \subseteq Y$, we write $A \cdot g \coloneqq \{a \cdot g \mid a \in A\}$.

\begin{definition}[\citet{conway-2013-course-in}, Example~1.11]\label{def:Haar-prior}
	Let \(G\) be a locally compact topological group. Then the \emph{right invariant
		Haar}
	measure (in short: right Haar measure) for \(G\) is a Borel measure \(\nu\)
	satisfying
	\begin{enumerate}
		\item $\nu(A) > 0$ for every nonempty open set $A \subseteq G$,
		\item $\nu(K) < \infty$ for every compact set $K \subseteq G$,
		\item\label{item:haar-measure-right-specification} $\nu(A \cdot g) = \nu(A)$ for
		every $g \in G$ and every measurable $A \subseteq G$.
	\end{enumerate}
\end{definition}

\section{Proofs Omitted from Main Text}
\label{app:proofs}
\begin{proof}{\bf [of Lemma~\ref{lem:caligali}]}
	Let \(A \subset {\mathbb R}_{> 0}\) be any Borel measurable set. In
	the equations below, the sum and integral can be swapped due to the
	monotone convergence theorem and the fact that $B_{\tau}$ is a
	positive function.
	%
	\begin{align*}
	\int_A \dif \mugeneral{P_1}{B}{\tau}
	&= \int_{\Omega} \ind[B_{\tau} \in A] \,\dif \mugeneral{P_1}{B}{\tau} \\
	&= \sum_{n=0}^{\infty} \int_{\samplespacen} \ind[B_{\tau} \in A] \ind[\tau=n]  \,\dif P^{(n)}_1 \\
	&\overset{(3)}{=} \sum_{n=0}^{\infty} \int_{\samplespacen} \ind[B_n \in A] \ind[\tau=n]
	\, P^{(n)}_1(\tau = n) \cdot \dif P^{(n)}_1(\cdot \mid \tau = n)\\
	&= \sum_{n=0}^{\infty} \int_{\samplespacen} \ind[B_n \in A]
	\, P^{(n)}_1(\tau = n) \cdot \dif P^{(n)}_1(\cdot \mid \tau = n) \\
	&\overset{(5)}{=} \sum_{n=0}^{\infty} \int_{r > 0} \ind[r \in A]
	\, P^{(n)}_1(\tau = n) \cdot \dif \mugencond{P_1}{B}{n}{\tau = n} \\
	&= \sum_{n=0}^{\infty} \int_{r > 0} \ind[r \in A]
	\rnd{\mugencond{P_1}{B}{n}{\tau = n}}{\mugencond{P_0}{B}{n}{\tau = n}} (r)
	\,P^{(n)}_1(\tau = n) \cdot \dif \mugencond{P_0}{B}{n}{\tau = n}  \\
	&\overset{(*)}{=} \sum_{n=0}^{\infty} \int_{r > 0} \ind[r \in A]
	\frac{P^{(n)}_0(\tau = n)}{P^{(n)}_1(\tau = n)} \cdot
	r \cdot P^{(n)}_1(\tau = n) \cdot \dif \mugencond{P_0}{B}{n}{\tau = n} \\
	&= \sum_{n=0}^{\infty} \int_{r > 0} \ind[r \in A]
	r \, P^{(n)}_0(\tau = n) \cdot \dif \mugencond{P_0}{B}{n}{\tau = n}  \\
	&= \sum_{n=0}^{\infty} \int_{\samplespacen} \ind[B_n \in A]
	\cdot B_n  \cdot P^{(n)}_0(\tau = n) \cdot \dif P_0^{(n)}(\cdot \mid \tau = n)  \\
	&= \sum_{n=0}^{\infty} \int_{\samplespacen} \ind[B_n \in A] \ind[\tau = n]
	\cdot B_n  \cdot P^{(n)}_0(\tau = n) \cdot \dif P_0^{(n)}(\cdot \mid \tau = n)   \\
	&= \sum_{n=0}^{\infty} \int_{\samplespacen} \ind[B_{\tau} \in A] \ind[\tau = n]
	\cdot B_n  \,\dif P^{(n)}_0  \\
	&= \int_{\Omega} \ind[B_{\tau} \in A] \left(\sum_{n=0}^{\infty} \ind[\tau=n]
	B_n \right) \,\dif P_0 \\
	&= \int_{\Omega} \ind[B_{\tau} \in A]  B_\tau \,\dif P_0\\
	&\overset{(14)}{=} \int_A t \mugeneral{P_0}{B}{\tau}(dt),
	\end{align*}
	where $(*)$ follows because of our fixed $n$-calibration
	assumption. Furthermore, (3) follows from the following
	equality for any $C \in \F$
	\begin{align}
	P^{(n)}_1(C \cap \{\tau = n \}) = P^{(n)}_1(\tau = n) \cdot P^{(n)}_1(C \mid \tau = n),
	\end{align}
	and in (5) we perform a change of variables where we integrate
	over the possible values of the Bayes Factor instead of over
	the outcome space, which we repeat in (14).
	
	We have shown that the function $g$ defined by \(g(t) = t\) is the Radon-Nikodym
	derivative \(\rnd{\mugeneral{P_1}{B}{\tau}}{\mugeneral{P_0}{B}{\tau}}\).
\end{proof}

\begin{proof}{\bf [of Lemma~\ref{lem:calibali}]}
	Let $A$ be any Borel subset of ${\mathbb R}_{> 0}$. We have:
	\begin{align*}
	\int_A \dif \mumargcond{1}{\gamma}{n}{x^m,\tau =n} & =
	\int_{\samplespacen} \ind[\gamma_n \in A] \dif \bar{P}^{(1)}_n(\cdot \mid x^m,\tau = n)\\
	& =
	\int_{\samplespacen} \ind[\gamma_n \in A] \left(\frac{\dif \bar{P}^{(1)}_n(\cdot \mid x^m,\tau = n)}{\dif \bar{P}^{(0)}_n(\cdot \mid x^m,\tau = n)} \right) \dif \bar{P}^{(0)}_n(\cdot \mid x^m,\tau = n)\\
	& \overset{(*)} {=}
	\int_{\samplespacen} \ind[\gamma_n \in A] \gamma_n \cdot
	\left(\frac{\pi(H_0 \mid x^m, \tau = n)}{\pi(H_1 \mid x^m, \tau = n)} \right) \dif \bar{P}^{(0)}_n(\cdot \mid x^m,\tau = n)\\
	& =
	\int_{\samplespacen} \ind[\gamma_n \in A] \gamma_n  \left(\frac{\bar{P}_0(\tau = n \mid x^m)\pi(H_0)}{
		\bar{P}_1(\tau = n  \mid x^m) \pi(H_1)} \right) \dif \bar{P}^{(0)}_n(\cdot \mid x^m,\tau = n)\\
	& =
	\left(\frac{\bar{P}_0(\tau = n \mid x^m)\pi(H_0 \mid x^m)}{
		\bar{P}_1(\tau = n  \mid x^m) \pi(H_1 \mid x^m)} \right)  \cdot \int_{A} \gamma_n
	\dif \mumargcond{0}{\gamma}{n}{x^m,\tau =n},
	\end{align*}
	where, for the case $m=0$,  $(*)$ follows from \eqref{eq:simplegivenn}, which can be verified to be still valid in our generalized setting. The case $m > 0$ follows in exactly the same way, by shifting the data by $m$ places (so that the new $x_1$ becomes what was $x_{m+1}$, and treating, for $k = 0,1$,  $\pi(H_k \mid x^m)$ as the priors for this shifted data problem, and then applying the above with $m=0$).
	
	We have shown that the Radon-Nikodym derivative
	$\rnd{
		\mumargcond{1}{\gamma}{n}{x^m}}{
		\mumargcond{0}{\gamma}{n}{x^m}}$ at $\gamma_n$
	is given by $\gamma_n \cdot \frac{\bar{P}_0(\tau = n \mid x^m)\pi(H_1 \mid x^m)}{
		\bar{P}_1(\tau = n  \mid x^m) \pi(H_0 \mid x^m)}$, which is what we had to show.
\end{proof}

\begin{proof}{\bf [of Lemma~\ref{lem:strongfixedn}]}
	Let $A'$ denote the event $V_n = 1$ and let
	$A \subset {\mathbb R}_{>0}$ be a Borel measurable subset of the
	positive real numbers.
	We have that $\beta_n$ is a function of the maximal invariant $U_n$
	as defined in Definition \ref{def:invariant}, and we write
	$\beta_n(U_n)$.
	With this notation, we have:
	\begin{align}
	&  \mug{1,g}{\beta}{n}(A \mid A')
	= \int_{\R_{>0}} \ind[A] \dif  \mugcond{1,g}{\beta}{n}{A'}
	\nonumber \\
	\nonumber
	&\overset{(2)}{=} \int_{\mathcal{U}_n} \ind[\beta_n(U_n) \in A] \dif \mugcond{1,g}{U}{n}{A'}\\
	&= \int_{\mathcal{U}_n} \ind[\beta_n(U_n) \in A]
	\frac{\dif \mugcond{1,g}{U}{n}{A'} }
	{\dif \mugcond{0,g}{U}{n}{A'} }  \,
	\dif \mugcond{0,g}{U}{n}{A'} \nonumber\\
	&\overset{(4)}{=} \int_{\mathcal{U}_n}  \ind[\beta_n(U^n) \in A]
	\frac{P_{0,g}^{(n)}(A')}{P_{1,g}^{(n)}(A')}\frac{\dif \mug{1,g}{U}{n} }
	{\dif \mug{0,g}{U}{n}}  \,
	\dif \mugcond{0,g}{U}{n}{A'} \nonumber\\
	&\overset{(5)}{=} \frac{P_{0,g}^{(n)}(A')}{P_{1,g}^{(n)}(A')} \cdot \int_{\mathcal{U}_n} \ind[\beta_n(U^n) \in A]       \beta_n(U_n) \,
	\dif \mugcond{0,g}{U}{n}{A'} \nonumber \\
	&= \frac{P_{0,g}^{(n)}(A')}{P_{1,g}^{(n)}(A')} \cdot \int_{\R_{>0}} \ind[A] t \,  \dif \mug{0,g}{\beta}{n}{A'}(t), \nonumber
	\end{align}
	where step~(2) holds because $\beta_n$ is $\G_n$-measurable.
	On the set $A'$ we have
	\begin{align*}
	\frac{\dif \mugcond{1,g}{U}{n}{A'} }
	{\dif \mugcond{0,g}{U}{n}{A'} }
	\frac{P_{1,g}^{(n)}(A')}{P_{0,g}^{(n)}(A')}
	&=
	\frac{\dif \mug{1,g}{U}{n} }
	{\dif \mug{0,g}{U}{n}},
	\end{align*}
	which explains step (4),
	and step~(5) follows from the definition of $\beta_n$ in Equation \eqref{eq:quotient-likelihood-ration}.
	
	We have shown that $\frac{P_{0,g}^{(n)}(A')}{P_{1,g}^{(n)}(A')} \cdot t$ is equal to the Radon-Nikodym derivative $\frac{\dif \mugcond{1,g}{\beta}{n}{V_n = 1}}
	{\dif \mugcond{0,g}{\beta}{n}{V_n= 1}}$, which is what we had to prove.
\end{proof}

\end{document}